\documentclass[final]{amsart}

\usepackage{amsfonts,amssymb,amsmath,amsthm}
\usepackage{mathtools}
\usepackage[utf8]{inputenc}
\usepackage{mathrsfs,bbm,stmaryrd}
\usepackage[color]{showkeys}
\definecolor{refkey}{rgb}{1,0,0} 
\definecolor{labelkey}{rgb}{0,0,1}

\usepackage{pstricks}
\usepackage{pstricks-add}

\usepackage[all]{xy}
\newcommand{\al}{\alpha}
\newcommand{\bt}{\beta}
\newcommand{\g}{\gamma}
\newcommand{\de}{\delta}
\newcommand{\e}{\varepsilon}

\newcommand{\tht}{\theta}

\newcommand{\la}{\lambda}

\newcommand{\s}{\sigma}

\newcommand{\om}{\omega}

\newcommand{\G}{\Gamma}
\newcommand{\D}{\Delta}

\newcommand{\Om}{\Omega}

\newcommand{\bbR}{\mathbb R} 

\newcommand{\bbP}{\mathbb P}
\newcommand{\bbE}{\mathbb E}

\renewcommand{\1}{\mathbbm{1}}

\newcommand{\cal}{\mathcal}

\newcommand{\cC}{\cal C}
\newcommand{\cD}{\cal D}
\newcommand{\cE}{\cal E}
\newcommand{\cF}{\cal F}

\newcommand{\cH}{\cal H}

\newcommand{\cL}{\cal L}

\newcommand{\cV}{\cal V}

\newcommand{\calc}{\mathscr}

\newcommand{\ccD}{\calc D}

\newcommand{\ccL}{\calc L}

\newcommand{\wt}{\widetilde}

\newtheorem{theorem}{Theorem}[section]
\newtheorem{lemma}[theorem]{Lemma}
\newtheorem{proposition}[theorem]{Proposition}

\newtheorem{example}{Example}

\theoremstyle{definition}
\newtheorem{definition}[theorem]{Definition}
\newtheorem{assumption}[theorem]{Assumptions}

\theoremstyle{remark}
\newtheorem{remark}{Remark}

\newcommand{\be}{\begin{equation}}
\newcommand{\ee}{\end{equation}}

\newcommand{\bea}{\be\begin{aligned}}
\newcommand{\eea}{\end{aligned}\ee}

\newcommand{\bal}{\begin{aligned}}
\newcommand{\eal}{\end{aligned}}

\newcommand \ba {\begin{array}}
\newcommand \ea {\end{array}}

\newcommand{\abs}[1]{\left\lvert{#1}\right\rvert}

\newcommand{\norm}[1]{\left\lVert{#1}\right\rVert}

\newcommand{\bra}[1]{\left\langle{#1}\right\rangle}

\newcommand{\pare}[1]{\left({#1}\right)}
\newcommand{\cro}[1]{\left[{#1}\right]}
\newcommand{\acc}[1]{\left\{{#1}\right\}}

\newcommand{\wconver}[3]{\overset{#3}{\xrightharpoonup[#1\to#2]{}}}

\newcommand{\conv}[2]{\xrightarrow[#1\to#2]{}}
\newcommand{\conver}[3]{\overset{#3}{\xrightarrow[#1\to#2]{}}}

\newcommand\grad{\nabla}
\newcommand\del{\partial}
\newcommand\dint{\mathrm{d}}

\newcommand\dx{\dint x}

\newcommand\ds{\dint s}
\newcommand\dt{\dint t}

\newcommand\dl{\dint l}

\newcommand\dm{\dint m}

\newcommand\supp{\mathrm{supp}\,}

\newcommand\barr[1]{\overline{#1}}

\newcommand{\vide}{\varnothing}
\newcommand{\seg}{\geqslant}
\newcommand{\ieg}{\leqslant}

\numberwithin{equation}{section}

\begin{document}

\title{Averaging principle for diffusion processes via Dirichlet forms}
\author{Florent Barret \and Max von Renesse}
\address{F. Barret, Universit\"at Leipzig, Fakultät für Mathematik und Informatik, Augustusplatz 10, 04109 Leipzig, Germany}
\email{barret@uni-leipzig.de}
\address{M. von Renesse, Universit\"at Leipzig, Fakultät für Mathematik und Informatik, Augustusplatz 10, 04109 Leipzig, Germany}
\email{renesse@uni-leipzig.de}

% \address{CMAP UMR 7641, \'Ecole Polytechnique CNRS, Route de Saclay,
% 91128 Palaiseau Cedex, France}
% \email{barret@cmap.polytechnique.fr}

\begin{abstract}
We study diffusion processes driven by a Brownian motion with regular drift in a finite dimension setting. The drift has two components on different time scales, a fast conservative component and a slow dissipative component. Using the theory of Dirichlet form and Mosco-convergence we obtain simpler proofs, interpretations and new results of the averaging principle for such processes when we speed up the conservative component. As a result, one obtains an effective process with values in the space of connected level sets of the conserved quantities. The use of Dirichlet forms provides a simple and nice way to characterize this process and its properties.
\end{abstract}

\maketitle

\date{\today}
\normalsize

\section{Introduction}

Our aim in this article is to introduce a new method, to obtain new results and to prove averaging principles for stochastic diffusion equations. We consider standard It\^o diffusions in a finite dimensional setting with drifts on different time scale. 
\medskip

We study the diffusion equation in $\bbR^{2}$, defined for $\al,\e>0$ by
\be\label{eq:diff.25}
\dint X_t=A\grad H(X_t)\dt-\al e(X_t)\dt+\sqrt{2\al\e}\dint B_t
\ee
where $A\grad$ is the symplectic gradient. $H$ is the Hamiltonian and $e$ is the drift and plays the role of a friction term (hence the minus sign in front of this term in Equation \eqref{eq:diff.25}). On the time scale $O\pare{\frac1{\al}}$, $Y_t=X_{\frac t{\al}}$ is a solution to the following equation:
\be\label{eq:diff.2}
\dint Y_t=\frac1{\al}A\grad H(X_t)\dt- e(X_t)\dt+\sqrt{2\e}\dint B_t.
\ee
We want to consider the limit $\al\to0$, for which we obtain an averaging along the orbits of the Hamiltonian system
\be
\dot{x}(t)=A\grad H(x(t)).
\ee
In fact, one can make the observation that since $H$ is preserved by the Hamiltonian vector field, the process $H(Y_t)$ satisfies
\be
H(Y_t)=H(Y_0)+\int_0^t-\grad H(Y_s)\cdot e(Y_s)+\e\D H(Y_s)\ds
+\int_0^t\grad H(Y_s)\cdot\dint B_s
\ee
which does not depend on $\al$. Then one could expect that the law of of the process $H(Y)$ will converge. However, we cannot obtain a convergence for $H(Y)$ itself if $H$ is not injective but for a projection of $Y$ on the ``orbit space" $\G$ defined as the space of connected level sets of $H$.
Previous convergence results have been obtained, notably by Freildin and Wentzell in several papers (see the new edition of their book \cite{freidlin.wentzell12} for a review, and the previous articles \cite{freidlin.wentzell93,freidlin.wentzell94}). 
\medskip

The main novelty of this paper concerns the method we use to prove this results and that it leads to new, more general results. Freidlin and Wentzell prove their results using the martingale formulation of the diffusion $H(Y)$. As $\al$ goes to $0$, they prove the convergence of the martingale problem and deduce the unicity of the limit and the equation satisfied by the solution. Existence is proven via tightness of the laws.

This paper is a nice application of the theory of convergence of Dirichlet form. We rely on the theory of non-symmetric Dirichlet forms (as exposed in the book \cite{ma.rockner92} by Ma and R\"ockner). We prove convergence, in a suitable sense, of a sequence of Dirichlet forms associated to the process defined in Equation \eqref{eq:diff.25}. As a consequence, we obtain convergence of the resolvants and therefore of the associated semigroups. In probabilistic terms, we prove convergence of the finite dimensional marginals of the process. Associated to the tightness of the sequence of laws, this proves the convergence in law.

We use results by Hino \cite{hino98}, and T\"olle \cite{toelle06} about the convergence of non-symmetric Dirichlet forms. Note that in our case, the functional space also changes along the convergence, and we refer to Kolesnikov \cite{kolesnikov05} who investigated such cases.
We refer also to Kuwae and Shioya \cite{kuwae.shioya03} for a quite complete exposition of spectral structures and their relations (semigroups, resolvants, Dirichlet forms and generators). 

The use of Dirichlet forms is very quick and simple, despite giving an indirect description of the limiting process (or its infinitesimal generator). However, with the use of the coarea Formula, we can recover the infinitesimal generator and its domain, therefore achieve a complete description of the process. 

We assume that we can define a nice Dirichlet form with a reference measure which is preserved under the accelerated flow. The other assumptions we make are quite general and could probably be greatly relaxed.

\medskip 

In this paper, we also prove convergence of the Dirichlet forms for higher dimensional diffusions on $\bbR^n$ defined by
\be
\dint Y_t=\frac1{\al}v(Y_t)\dt+u(Y_t)\dt+\sqrt{2\e}\s(Y_t)\dint B_t.
\ee
We consider the limit of the Dirichlet forms, as $\al$ goes to $0$, for the process $G(Y)$ where $G$ is a $\bbR^m$-valued function (with $m<n$) conserved along the flow of the vector field $v$. We only sketch the computation of the limiting infinitesimal generator since a complete computation would need a very detailed description of the space $\G$ (defined as the quotient space for the connected level sets of $G$) which is beyond the scope of this article.

Averaging principles for higher dimensional diffusions have been investigated in several papers. In \cite{freidlin.wentzell04}, Freidlin and Wentzell show an averaging principle for a diffusion where the fast component concerns only the first two coordinates. In \cite{freidlin.weber04} by Freidlin and Weber, the fast component contains a Brownian term and there is only one first integral ($m=1$ in our notation). In \cite{freidlin.weber01}, the same authors investigate the perturbation of an Hamiltonian system with only one first integral but draws conclusions on the PDE counterpart of the averaging.

\medskip 

We give now more details for the Hamiltonian two-dimensional case. We choose a reference measure $\mu$ and define the (pre-)Dirichlet form, denoted  $E_\al$, associated to the infinitesimal generator, denoted $L_\al$, of the diffusion given by Equation \ref{eq:diff.2} in $L^2(\mu)$: for $f,g$ two $C^2$ functions with compact support
\be
E_\al(f,g)=-\bra{L_\al f,g}_{L^2(\mu)}=-\int_{\bbR^2}L_\al fg\dint\mu.
\ee
Under suitable conditions on $\mu$, $E_\al$ is a Dirichlet form and characterizes completely the infinitesimal generator $L_\al$.

The projected Dirichlet form, denoted $\cE$, is constructed by restricting the set of test functions : we consider functions constant on connected level sets of $H$. We do it rigorously by considering the connected level sets of $H$ as equivalence classes. We denote $\G$ the quotient set, and $\pi$ the canonical projection onto $\G$. Naturally, we can associate a $L^2(\G)$ function space to the space $\G$ (containing functions $f$ defined on $\G$ such that $f\circ\pi$ is in $L^2(\mu)$). The projected Dirichlet form $\cE$ is defined as: for $f,g$
\be\label{eq:intro.0}
\cE(f,g)=E_\al(f\circ\pi,g\circ\pi).
\ee
Due to the choice of $\mu$, we prove, and this is the most important remark, that $\cE$ does not depend on $\al$ and is, in itself, a nice Dirichlet form on $L^2(\G)$.

Moreover, we prove that, in a Mosco-convergence sense, the sequence of Dirichlet forms $E_\al$ and their domains, converges to $\cE$. The convergence in law of the process follows by using the tightness.

\medskip

The convergence of Dirichlet forms is quite abstract but powerful since it can be applied to very general cases (see Section \ref{sec:general}). However, in order to have a more intuitive description of the limiting process, we have to write the infinitesimal operator $\cL$ (and its domain) associated to the Dirichlet form as 
\be\label{eq:intro.1}
\cE(f,g)=-\bra{\cL f,g}_{L^2(\G)}=-\int_{\bbR^2}(\cL f)\circ\pi g\circ\pi\dint \mu.
\ee

To this aim, we need a better understanding of the space $\G$, it can be easily done in the Hamiltonian case on $\bbR^2$, but it is much more involved in higher dimensions. 

In $\bbR^2$, the space $\G$ is a graph with vertices and edges, on each edge the averaging process is a classical diffusion whose drift and diffusion coefficients could be easily computed. However, at a vertex (a point gluing together several edges), we obtain a so-called gluing condition giving a restriction on the domain of the operator and therefore on the behavior of the process when (or if) it reaches this vertex. 

The limiting diffusion is therefore a process on a graph and is described by 
\begin{itemize}
\item an infinitesimal generator on each edge (a second order differential operator);
\item a gluing condition at each vertex.
\end{itemize}
Analysis of such processes could be done based on one-dimensional diffusions (see Feller \cite{feller54} or Mandl \cite{mandl68}), we also cite the work by Kant, Klauss, Voigt and Weber \cite{KKVW09} which investigates such processes from a Dirichlet-form point of view and more recently the work of Kostrykin, Potthoff and Schrader \cite{KPS12}. Large deviations for diffusions process on graphs have been proven in \cite{freidlin.sheu00}.

The infinitesimal generator is solution of Equation \eqref{eq:intro.1}. We compute the infinitesimal generator in two steps: 
\begin{enumerate}
\item using the coarea Formula, we compute the measure $\pi_*\mu$ on $\G$ (projection of $\mu$ on $\G$) and we write the Dirichlet form as an integral on $\G$;
\item we use a integration by part on each edge to transfer the derivatives on $g$ to $f$.
\end{enumerate}
The identification of the two sides of Equation \eqref{eq:intro.1} is made separately on the edges and the vertices on $\G$. On the edges, we obtain a second-order differential operator:
\be
\cL u=au''+bu'
\ee
whereas at a vertex $O$, we obtain a gluing relation:
\be
-\bt\cL u(O)=\g u(O)+\sum_{i}\al_iD_iu(O)
\ee
where $\g$, $\bt$, $\al_i$ are constants, the sum is made on all edges incident to the vertex $i$ and $D_iu$ is the one-sided derivative of $u$ at $O$ along the edge $i$.

\medskip 

The remaining part of this paper is organized as follows. In Section \ref{sec:2d} we present the Hamiltonian case in $\bbR^2$, the main assumptions, we define the Dirichlet form $E_\al$ and the orbit space $\G$. In Section \ref{sec:conv}, we prove the convergence in law of the projected process to a limiting process defined by the projected Dirichlet form $\cE$. In Section \ref{sec:iden}, we compute the infinitesimal generator and draw some consequences about the behavior of the limiting process on $\G$. 
Lastly, in the last section (Section \ref{sec:general}), we generalize our method to diffusions in arbitrary dimension.

\paragraph{\textbf{Acknowledgements}} We thank an anonymous referee for pointing out the fact that condition \eqref{eq:supermedian} of Assumption \ref{ass.2} could probably be relaxed to $\grad\cdot(hF)\ieg c$ for some positive constant $c$. In this case, one should work with lower bounded semi-Dirichlet forms (see e.g. \cite{oshima13}). However, we ask for condition \eqref{eq:supermedian} in order to work with simple Dirichlet forms (and thus simplify the Mosco-convergence results).

\section{Two dimensional case with additive noise}\label{sec:2d}
\subsection{Properties of the process}

We consider the solution $(Y_t)$ of equation \eqref{eq:diff.2}. 
\begin{assumption}\label{ass.1}
We assume that $e$ is a $C^1$ bounded vector field. We suppose that $H$ is $C^2$, bounded from below and has compact level sets. We also assume that $H$ has bounded second derivatives.
\end{assumption}

\begin{remark}
The assumption on $e$ (and the fact that $H$ must have bounded second derivatives) ensures the existence of a solution $Y$ to Equation \eqref{eq:diff.2}.
The boundedness of the second derivative of $H$ is a technical assumption (also made in \cite{freidlin.wentzell12}) which ensures an easy proof of the tightness of the process.
%The additional condition is made for simplicity in this first part and ensures that the Gibbs measure $\mu_e$ associated to $E$ ($\dint \mu=e^{-E\e}\dx$) is invariant for the process $Y$, for all $\al>0$.
\end{remark}

We follow Ma-Rockner \cite{ma.rockner92}. The Dirichlet form is defined through the infinitesimal operator $L_\al$, which is a closed, densely defined operator. We define this operator and its domain $\cD(L_\al)$ via the transition semigroup of Equation \eqref{eq:diff.2}:
\begin{align}
\cD(L_\al)
&=\acc{f\in C_b(\bbR^2), t^{-1}(\bbE_{\cdot}[f(Y_t)]-f(\cdot)) \text{ converges uniformly as $t\to0$}}\\
L_{\al}f
&=\frac1{\al}A\grad H\cdot\grad f- e\cdot\grad f+\e\D f\text{ for $f\in C^2_c(\bbR^2)$.}
\end{align}
We consider a measure $\dint \mu=h(x)\dx$, where $h$ is $C^2$ and strictly positive.
We define the Hilbert spaces $L^2(\mu)$ and $H^1(\mu)$ as the weighted $L^2$ and $H^1$ sets with their scalar products $\bra{\cdot,\cdot}_{L^2(\mu)}$ and $\bra{\cdot,\cdot}_{H^1(\mu)}$:
\begin{align}
L^2(\mu)&=\acc{f, \bra{f,f}_{L^2(\mu)}=\int f^2\dint \mu<+\infty}\\
H^1(\mu)&=\acc{f, \bra{f,f}_{H^1(\mu)}=\int f^2+\abs{\grad f}^2\dint \mu<+\infty}.
\end{align}

Let us define the vector field $F$ as
\be\label{eq:F}
F=e+\frac{\e}{h}\grad h.
\ee

We have first the following Lemma
\begin{lemma}
Let us consider the bilinear form, for $f,g\in\cD\cap C_c(\bbR^2)$
\begin{align}
E_{\al}(f,g)&=-\bra{L_\al f,g}_{L^2(\mu)}
\end{align}
The bilinear form $E_\al$ can be uniquely decomposed in two parts: one symmetric, $E^s_{\al}$, and one antisymmetric, $E^a_{\al}$:
\begin{align}\nonumber
E_{\al}(f,g)
&=E^s_{\al}(f,g)+E^a_{\al}(f,g)\\\label{eq:DFsym}
E^s_{\al}(f,g)
&=\e\int\grad f\cdot\grad g\dint \mu
-\frac12\int \frac1h \grad\cdot(hF) fg\dint \mu
+\frac1{2\al}\int\frac1h\grad\cdot(hA\grad H)fg\dint \mu\\\label{eq:DFantisym}
E^a_{\al}(f,g)
&=-\frac1{2\al}\int_{\bbR^2}A\grad H\cdot\cro{g\grad f-f \grad g}\dint \mu
+\frac12\int F\cdot \cro{\grad f g-\grad g f}\dint \mu.
\end{align}
\end{lemma}

\begin{proof}
The lemma follows from an integration by parts.
For $f,g\in C^2_c(\bbR^2)$, we get
\be\label{eq:int.1}
\int \D f g\dint\mu =-\int \grad f \cdot \grad (g h)\dx 
=-\int \grad f\cdot\grad g\dint \mu-\int \frac1h\grad h\cdot \grad f g\dint \mu.
\ee
Also, for any regular vector field $G\in C^2(\bbR^2, \bbR^2)$, we have
\begin{align}\label{eq:int.2}
\int G\cdot \grad f g\dint \mu
&=\frac12\int G\cdot (\grad f g-\grad g f)\dint \mu-\frac12\int\frac1h\grad \cdot (hG)fg\dint \mu.
\end{align}

Using Equations \eqref{eq:int.1} and \eqref{eq:int.2}, we obtain Equations \eqref{eq:DFsym} and \eqref{eq:DFantisym} for $f,g$ sufficiently regular with compact support. By density of smooth functions in $\cD\cap C_c(\bbR^2)$, we get the lemma.
\end{proof}

Using Ma-Rockner, we will extend the bilinear form $E_\al$ as a Dirichlet form. We denote
\begin{align}\label{eq:N1}
E^1_{\al}(f,g)&=E_\al(f,g)+\bra{f,g}_{L^2(\mu)}\\\label{eq:N2}
E^{s,1}_{\al}(f,g)&=E^s_\al(f,g)+\bra{f,g}_{L^2(\mu)}
\end{align}
Let us recall (\cite{ma.rockner92} Definition 4.5 p.34) that the bilinear form $E_\al$ with a domain $\ccD$ dense in $L^2(\mu)$ is a Dirichlet form if:
\begin{itemize}
\item $E^{s}_\al$ is positive definite on $\ccD$;
\item $(E_\al,\ccD)$ is closed (i.e. $\ccD$ is closed and complete w.r.t. $E^{s,1}_\al$, or equivalently $(\ccD,E^{s,1}_\al)$ is a Hilbert space);
\item $(E_\al,\ccD)$ is coercive i.e. there exists $K>0$ such that for all $f,g\in\ccD$
\be
\abs{E^1_\al(f,g)}\ieg  K E^{s,1}_\al(f,f)^{1/2} E^{s,1}_\al(g,g)^{1/2};
\ee
\item for all $u\in\ccD$, we have the contraction properties:
\begin{align}\label{eq:contractions}
\min(u_+,1)\in \ccD
&&E_{\al}(u+\min(u_+,1),u-\min(u_+,1))&\seg0\\\nonumber
&&E_{\al}(u-\min(u_+,1),u+\min(u_+,1))&\seg0
\end{align}
where $u_+=\max(u,0)$ is the positive part of $u$.
\end{itemize}
A Dirichlet form is said to satisfy the local property if for any $f,g\in\ccD$ such that $\supp f\cap \supp g=\vide$, $E_\al(f,g)=0$. The Dirichlet form is said to be regular if $C_c\cap\ccD$ is dense in $(\ccD,E^{s,1}_\al)$ and dense in $C_c$ with respect to the uniform norm.

We complete our set of assumptions.
\begin{assumption}\label{ass.2}
We assume that there exists a function $h$ constant on connected level sets of $h$ such that 
\be\label{eq:supermedian}
\grad\cdot(hF)\ieg 0
\ee
i.e. the $h$-divergence of $F$ is non positive ($F$ is defined by Equation \eqref{eq:F}).
We also assume that the vector field $F$ is bounded on $\bbR^2$, and that its divergence (with respect to $h$, i.e. $h^{-1}\grad \cdot(hF)$) is also bounded.
%Let us also assume that $\abs{\grad H}$ is in $L^2(\mu)$.
\end{assumption}

We make several remarks on these assumptions.

\begin{remark}
These assumptions allow us to extend the bilinear form $E_\al$ as a Dirichlet form. We need the condition \eqref{eq:supermedian} to ensure that $E^{s,1}_\al$ is a positive bilinear form. The condition that $h$ is constant on connected level sets ensures us that the symmetric part \eqref{eq:DFsym} does not depend on $\al$. Lastly, the symmetric form $E^{s,1}_\al$ is a norm equivalent to the $H^1(\mu)$ norm thanks to the assumption that $F$ and its $h$-divergence are bounded.
\end{remark}

\begin{remark}
Note that Condition \eqref{eq:supermedian} simplifies in
\be
\e\D h+\grad\cdot(he)=\e\D h+\grad h\cdot e+h\grad\cdot e\ieg 0.
\ee
Therefore if $\grad\cdot e\ieg 0$, we can choose $h$ to be constant and then $\mu$ can be the Lebesgue measure. 
Moreover, our assumption implies a necessary condition on $e$. In fact, we ask for $h$ to be constant on connected level set, therefore, on the interior of $\acc{\grad H=0}$, $e$ must satisfies $\grad\cdot e\ieg 0$.
Finally, note that if $e=\grad f$ where $f$ is a real-valued potential constant on connected level sets of $H$, one can choose the Gibbs measure i.e. $h=e^{-f/\e}$.
\end{remark}

\begin{remark}
Let us also remark that the adjoint $L^*_\al$ (w.r.t. the usual scalar product in $L^2$) of $L_\al$ is, for $f\in C^2_c$
\be
L^*_\al f=\e\D f+\grad\cdot(f e)-\frac1{\al}\grad\cdot(fA\grad H).
\ee
Thus the fact that $\grad h\cdot A\grad H=0$ (since $h$ is constant on connected level sets) and the condition \eqref{eq:supermedian} implies that
\be
L^*_\al h
=\e\D h+\grad\cdot(h e)-\frac1{\al}\grad\cdot(hA\grad H)
=\e\D h+\grad\cdot(h e)=\grad\cdot(hF)\ieg0.
\ee
Therefore $h$ must be the density of a supermedian measure for the transitions probabilities of Equation \eqref{eq:diff.2} (cf. \cite{ma.rockner92} p.62) i.e. for any $t\seg 0$, and $f\in C_b(\bbR^2)$
\be
\bbE_\mu[f(Y_t)]=\int \bbE_x[f(Y_t)]\dint \mu(x)\ieg\int f\dint \mu.
\ee
Note that this is precisely the condition needed in order to define a Dirichlet form from a diffusion process (\cite{ma.rockner92} Section IV.2 p.92).
\end{remark}

\begin{proposition}
Under the set of assumptions \ref{ass.1} and \ref{ass.2}, $E_{\al}$ is coercive, closable and can be extended into a Dirichlet form, also denoted $E_{\al}$, on $L^2(\mu)$ with domain $\ccD(E_\al)=H^1(\mu)$ defined as the completion in $L^2(\mu)$ of $\ccD(L_\al)$.
Moreover $E_\al$ is regular and has the local property. The Dirichlet form is then $E_{\al}(f,g)=E^s_{\al}(f,g)+E^a_{\al}(f,g)$
\begin{align}\label{eq:DFsym.2}
E^s_{\al}(f,g)
&=\e\int\grad f\cdot\grad g\dint \mu
-\frac12\int \frac1h \grad\cdot(hF) fg\dint \mu\\\label{eq:DFantisym.2}
E^a_{\al}(f,g)
&=-\frac1{2\al}\int_{\bbR^2}A\grad H\cdot\cro{g\grad f-f \grad g}\dint \mu
+\frac12\int F\cdot \cro{\grad f g-\grad g f}\dint \mu.
\end{align}
\end{proposition}

\begin{proof}
The domain $\ccD(E_\al)$ is defined as the completion of $\ccD(L_\al)$ with respect to the positive bilinear form $E_\al^{s,1}$ and it defines a coercive closed bilinear form (\cite{ma.rockner92} Theorem 2.15 p.22). The contraction properties \eqref{eq:contractions} are proved in \cite{ma.rockner92} Section II.2.d. (p.48) under some conditions (Equations (2.16) p.48) which come as a consequence of our assumption \eqref{eq:supermedian}.

The fact that the domain is $H^1(\mu)$ comes from the fact that, under the assumption \ref{ass.2}, $F$ and $\grad\cdot(hF)$ are bounded, $E_\al^{s,1}$ and $\norm{\cdot}_{H^1(\mu)}$ are equivalent norms. The regularity is obvious from the properties of the sets $H^1(\mu)$ and $C_c$. The local property is also obvious from Equations \eqref{eq:DFsym} and \eqref{eq:DFantisym}.
\end{proof}

\subsection{Orbit space}

We need to construct our new state space from the Hamiltonian $H$. We define a graph, with a set of vertices $\cV$, and edges (with their length).
In the rest of the section we denote $\cC(x)$ the connected level set of $H$ containing $x$.

We define a vertex $O$ as a connected level set of $H$ containing a stationary point. We denote $\cV$ this set to which we add a artificial vertex $O=\infty$. Let us recall that we assume that $H$ is bounded below (Assumption \ref{ass.1}).

In order to define the edges, we prove a lemma.
\begin{lemma}\label{lem:edge}
Given any $x\in\bbR^2$ such that $\cC(x)$ does not contain any stationary point (i.e. $\cC(x)$ is not a vertex), there exists a unique closed interval $I$ of $\bbR$ of the form $I=[m_-,m_+]$ or $I=[m_-,+\infty[$ such that $x\in H^{-1}(\mathring I)$ and
the connected domain $A_I$ of $H^{-1}(I)$ containing $x$ satisfies
\begin{itemize}
\item $H^{-1}(m_-)\cap A_I$ is a vertex, denoted $O_I^-$;
\item if $m_+<+\infty$, $H^{-1}(m_+)\cap A_I$ is also a vertex denoted $O_I^+$ (with the convention $O_I^+=\infty$ if $m_+=+\infty$).
\end{itemize}
Moreover for any $m\in I$, $A_I\cap H^{-1}(m)$ is a connected level set of $H$.
\end{lemma}

\begin{proof}
Let us consider $x\in\bbR^2$ such that $\cC(x)$ does not contain a stationary point, and denote $m_0=H(x)$. 
For $a<m_0$ and $b>m_0$, let us denote $C(a,b)$ the connected domain of $\acc{x, H(x)\in ]a,b[}$ containing $\cC(x)$. 
Note that since $H$ is $C^1$ and by assumption $\cC(x)$ is compact, there exists $\eta>0$ such that $C(m_0-\eta, m_0+\eta)$ does not contain any stationary point.
Then let us denote, 
\begin{align}
m_-&=\min\acc{ a<m_0, \forall z\in C(a,m_0+\eta), \grad H(z)\neq 0}\\
m_+&=\max\acc{ b>m_0, \forall z\in C(m_0-\eta,b), \grad H(z)\neq 0}\ieg+\infty.
\end{align}
We have $m_-\ieg m_0-\eta$ and $m_+\seg m_0+\eta$. 
Then, we see that $I=[m_-,m_+]$ (or $I=[m_-,+\infty[$ if $m_+=+\infty$) satisfies our properties. The facts that $x\in H^{-1}(\mathring{I})$ and that $H^{-1}(m_-)\cap A_I$ (resp. $H^{-1}(m_+)\cap A_I$) contains a vertex are obvious.

To prove that for all $m\in I$, $A_I\cap H^{-1}(m)$ is a connected level set, we consider the flow $\phi_t$ induced by the differential equation
\be
\dot{y}(t)=\frac{\grad H(y(t))}{\abs{\grad H (y(t))}^2}.
\ee
Remark that this flow is well defined until the orbit reach a point for which $\grad H(x)=0$. Then it is well defined on $A_I\cap H^{-1}(\mathring{I})$ and that, for all $y$ in this set, $H(\phi_t(y))=t+H(y)$, for $t$ such that $\phi_s(y)$ is not a stationnary point for any $s$. Then $\phi_t(\cC(x))$ is a connected level set of $H$ for all $t\in \mathring{I}$ and by definition of $A_I$, we get $A_I\cap H^{-1}(m)=\phi_{m-m_0}(\cC(x))$. Therefore, it is a connected set for all $m\in\mathring{I}$.
\end{proof}

\begin{remark}
Given a point $x$ and the interval $I$ associated to it by Lemma \ref{lem:edge}, for $y\in A_I\cap H^{-1}(\mathring{I})$, the interval associated by the lemma is also $I$. For all $x$, we denote $(I(x),O^-_I(x),O^+_I(x))$ the interval and the two vertices associated to $x$ through Lemma \ref{lem:edge}.
Under the Assumption \ref{ass.1}, the set $S=(I_i,O^-_i,O^+_i)_i$ of distinct triplets given by Lemma \ref{lem:edge} is countable, therefore the set of indices $i$ is countable.
\end{remark}

We are ready to define our graph.
\begin{definition}
Let consider the set $S=(I_i,O^-_i,O^+_i)_i$.
Our graph $\G$, is given by the collection of edges $I_i$, the collection of vertices $\cV$. An edge $I$ is related to the vertices $O^-, O^+$ such that $(I,O^-,O^+)\in S$. 

We also define the projection $\pi$ from $\bbR^2$ to $\G$. For $x\in\bbR^2$, we define  $\pi(x)=O$ if $\cC(x)$ is the vertex $O$, otherwise $\pi(x)=(H(x),i(x))$ where $i(x)$ is defined as the index such that $I(x)=I_{i(x)}$.

We equip $\G$ with the minimal topology making $\pi$ continuous.
\end{definition}

\begin{remark}
The space $\G$ is therefore a disjoint countable union of interval $(I_i)_i$ of $\bbR$ glued together by one or two of their extremities $(O_k)$ associated to stationnary points of $H$. $i(x)$ is a discrete first integral for the system but is defined only in the interior of the edges. At a vertex we can choose one the indices of the incident edges (e.g. the lowest integer). Note also that $\G$ is a tree (i.e. it does not have any cycle).
\end{remark}

Let us now consider the equivalence relation $\sim$ on $\bbR^2$ defined by
\be
x\sim y\Leftrightarrow \text{$x$ and $y$ are in the same connected component of a level set of $H$}.
\ee

\begin{proposition}\label{prop:orbit}
We have
\be
\G=\bbR^2/\sim.
\ee
\end{proposition}

\begin{proof}
We construct a bijection $\phi$ from $\bbR^2/\sim$ to $\G$. Given a connected level set in $\cC$ in $\bbR^2/\sim$, if there is a stationary point $x\in\cC$, then by definition of $\G$, $\cC$ is a vertex and $\phi(\cC)=\cC\in\cV$. If $\cC$ does not contain any stationary point, then, according to Lemma \ref{lem:edge}, there exists a unique edge $I_i$ and $\phi(\cC)=(H(\cC),i)$ where $H(\cC)$ is the common value of $H$ along $\cC$. $\phi^{-1}$ is also well defined since, according to \ref{lem:edge}, for any $i$ and $h\in\mathring{I}$, the set $\pi^{-1}((h,i))$ is a connected level set of $H$, therefore a unique equivalence class in $\bbR^2/\sim$.
\end{proof}

We define $C_i(m)$ the connected level set 
\be
C_i(m)=\acc{x, H(x)=m, i(x)=i}=\acc{x, \pi(x)=(m,i)}.
\ee
$A_i$ is the domain of $\bbR^2$
\be\label{eq:defset}
A_i=\acc{x, i(x)=i, H(x)\in\mathring{I_i}}=\bigcup_{m\in \mathring{I_i}}C_i(m).
\ee

\begin{remark}
Let $h$ be a smooth function constant on connected level sets (such as the density of the measure $\mu$ satisfying Assumption \ref{ass.2}). Note that on each $A_i$, there exists a function $\psi_i$ such that $h=\psi_i(H)$. Then, we get, on each $A_i$,
\be
\grad\cdot(h A\grad H)=\grad\cdot(\psi_i(H) A\grad H)=0.
\ee
Since for a point $x$ which is not in any $A_i$, $\grad H(x)=0$, we have $\grad\cdot(h A\grad H)=0$ on the whole space $\bbR^2$.
\end{remark}

We introduce a decomposition of $\bbR^2$ into disjoints sets. Let us first introduce a partition of the vertices:
\begin{align}
\cV_*&=\acc{O\in \cV, \mu(\pi^{-1}(O))>0}&
\cV_0&=\acc{O\in \cV, \mu(\pi^{-1}(O))=0}.
\end{align}

\begin{lemma}\label{lem:decomp}
We have the following decompositions, and for $f\in L^1(\mu)$
\begin{align}\label{eq:decomp}
\bbR^2
&=\bigcup_i A_i\cup\bigcup_{O\in\cV_*}\pi^{-1}(O)\cup\bigcup_{O\in\cV_0}\pi^{-1}(O)\\\label{eq:integral}
\int_{\bbR^2}f\dint \mu
&=\sum_i\int_{A_i}f\dint \mu+\sum_{O\in\cV_*}\int_{\pi^{-1}(O)}f\dint \mu.
\end{align}
We also have $\grad H=0$, identically on $\pi^{-1}(O)$, for $O\in \cV_*$.
\end{lemma}

% \begin{proof}
% By definition, each vertex is either in $\cV_0$ or in $\cV_+$, and only the vertices which correspond to a set of positive measure can contribute to the integral \eqref{eq:integral}. 
% 
% For $O\in\cV_+$, by definition $\pi^{-1}(O)$ is closed and must be of positive (Lebesgue) measure since $\mu$ has a density with respect to the Lebesgue measure. Also, $H$ is constant on $\pi^{-1}(O)$. Then, $\grad H$ is identically vanishing on $\mathring{\pi^{-1}(O)}$ which is not empty since it has positive Lebesgue measure. Therefore, by continuity of $\grad H$, it must vanish on the closure of this set which is precisely $\pi^{-1}(O)$.
% \end{proof}

\subsubsection{Examples}
We give in this section some examples of the space $\G$ obtain for some given $H$ and some examples of vertices. In the following we denote, for $(x_1,x_2)\in\bbR^2$, $r=x_1^2+x_2^2$.

\begin{example}
The simplest example is maybe given by 
\be\nonumber
H_1(x_1,x_2)=\frac{r^2}2
\qquad
\begin{aligned}
\xymatrix{
&&+\infty\\
&&\\
0\ar@{->}[uu]^{H_1}&&(0,0)\ar@{-}[uu]}
\end{aligned}
\ee
The space $\G$ has a unique edge $I_1=\bbR_+$ and one vertex $O=(0,0)$. The connected level set associated to $H_1=m\in I_1$ is the circle centered on $(0,0)$ with radius $\sqrt{2m}$. The vertex $O$ is a simple point. 

Note that if we choose for $H$:
\be
H_1(x_1,x_2)=
\begin{cases}
0& \text{if $r\ieg 1$}\\
(r-1)^2& \text{otherwise.}
\end{cases}
\ee
The orbit space $\G$ is the same but the vertex $O$ is the whole ball of radius $1$. This is an example of a vertex having some positive mass.
\end{example}

\begin{example}
Let us consider the function $H_2$ defined by $H_2(x_1,x_2)=\frac{x^4_1}4-\frac{x^2_1}2+\frac{x^2_2}2$.
\begin{align}\nonumber
\xymatrix{
&&&\infty&\\
H_2=0\ar@{->}[u]&&&(0,0)\ar@{-}[u]&\\
H_2=-\frac14\ar@{->}[u]&&(-1,0)\ar@{-}[ur]&&(1,0)\ar@{-}[ul]}
\end{align}
The space $\G$ is composed of three edges glued together at a point representing the connected level set of the saddle $(0,0)$.
This connected level set is the internal vertex and does not have any mass. However, as in the first example, the vertex could have some mass (see Figure \ref{fig:1} and Figure \ref{fig:2}).
\begin{figure}[h!]
\begin{minipage}{0.48\textwidth}
\centering
\psset{unit=40pt}
\begin{pspicture}(-2,-1)(2,1)
   \psaxes{->}(0,0)(-2,-1)(2,1)
   \psgrid[griddots=5, subgriddiv=0, gridlabels=0pt](-2,-1)(2,1)
   \psecurve[linewidth=1pt](-1,-0.707)(0,0)(1,0.707)(1.414,0)(1,-0.707)(0,0)(-1,0.707)(-1.414,0)(-1,-0.707)(0,0)(1,0.707)
\end{pspicture}
\caption{\label{fig:1}Connected level set associated to the internal vertex for $H_2$.}
\end{minipage}\quad%
\begin{minipage}{0.48\textwidth}
\centering
\psset{unit=40pt}
\begin{pspicture}(-2,-1)(2,1)
   \psaxes{->}(0,0)(-2,-1)(2,1)
   \psgrid[griddots=5, subgriddiv=0, gridlabels=0pt](-2,-1)(2,1)
   \pscurve[linewidth=1pt](0.5,0.707)(1,0.707)(1.414,0)(1,-0.707)(0.5,-0.707)
   \pscurve[linewidth=1pt](-0.5,0.707)(-1,0.707)(-1.414,0)(-1,-0.707)(-0.5,-0.707)
   \psframe*(-0.5,-0.707)(0.5,0.707)
\end{pspicture}
\caption{\label{fig:2}Connected level set associated to an internal vertex with some mass.}
\end{minipage}
\end{figure}
\end{example}

\begin{example}
Let us consider the function $H_3$ defined by
$H_3(x_1,x_2)=\frac{x^4_1}4-\frac{x^2_1}2+\frac{x^4_2}4-\frac{x^2_2}2$. The space $\G$ is composed of five branches glued together at a point $O_1$ representing the connected level set of the saddles $(\pm1,0),(0,\pm1)$:
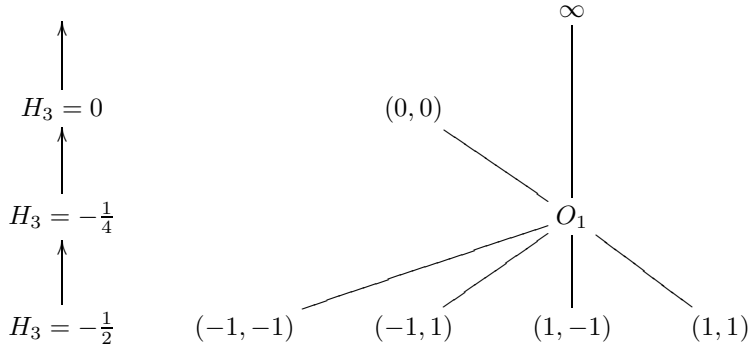
\be\nonumber
\begin{aligned}
\xymatrix{
&&&\infty&\\
H_3=0\ar@{->}[u]&&(0,0)\ar@{-}[dr]&&\\
H_3=-\frac14\ar@{->}[u]&&&O_1\ar@{-}[uu]&\\
H_3=-\frac12\ar@{->}[u]&(-1,-1)\ar@{-}[urr]&(-1,1)\ar@{-}[ur]&(1,-1)\ar@{-}[u]&(1,1)\ar@{-}[ul]}
\end{aligned}
\ee
In this case the internal vertex at $H_3=-1/4$ has a more complex structure.
\begin{figure}[ht]
\centering
\psset{unit=50pt}
\begin{pspicture}(-2,-2)(2,2)
   \psaxes{->}(0,0)(-2,-2)(2,2)
   \psgrid[griddots=5, subgriddiv=0, gridlabels=0pt](-2,-2)(2,2)
   \rput{45}{%
   \psellipse[linewidth=1pt](0,0)(1.848,0.765)%
   \psellipse[linewidth=1pt](0,0)(0.765,1.848)}
\end{pspicture}
\caption{Connected level set associated to the internal vertex for $H_3$.}
\end{figure}
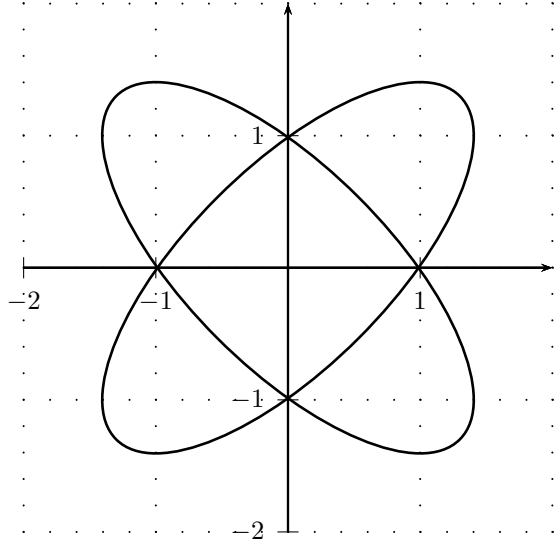

\end{example}

\section{Convergence of the Process}\label{sec:conv}

%\subsection{Mosco convergence of non-symmetric Dirichlet forms}

\subsection{Projected Dirichlet form}

We make the following remark: we see that if we evaluate  $E_{\al}$ on functions which are constant on connected level sets of $H$, then the $\al$-dependence in the antisymmetric part (Equation \eqref{eq:DFantisym}) of the Dirichlet form vanishes. More formally, the projection $\pi$ defines a pull-back $\pi_*$ on the functions on $\G$ :
\begin{align}\nonumber
\pi_*: \bbR^{\G}&\to \bbR^{\bbR^2}\\
u&\mapsto \barr u=u\circ \pi=u(H,i)
\end{align}

The space $\G$ has a topological structure (induced by $\pi$).
We define the $L^2$ and $H^1$ Hilbert space by using this pull-back:
\begin{align}
L^2(\G)&=\pi_*^{-1}(L^2(\mu))=\acc{u, \barr u\in L^2(\mu)}\\
\bra{u,v}_{L^2(\G)}&=\bra{\barr u,\barr v}_{L^2(\mu)}=\int_{\bbR^2}\barr u \,\barr v\dint \mu\\
H^1(\G)&=\pi_*^{-1}(H^1(\mu))=\acc{u,\barr u\in H^1(\mu)}\\
\bra{u,v}_{H^1(\G)}&=\bra{\barr u,\barr v}_{H^1(\mu)}=\int_{\bbR^2}\barr u\,\barr v+\grad\barr u\cdot \grad\barr v\dint \mu.
\end{align}
Let us denote $\Pi$ the subset of functions on $\bbR^2$ which are constant on connected level set: $\Pi=\pi_*(\bbR^{\G})$.

The pull-back naturally provides a identification (since it is an isometry by definition of $L^2(\G)$ and $H^1(\G)$):
\begin{align}
L^2(\G)&\simeq L^2(\mu)\cap\Pi&
H^1(\G)&\simeq H^1(\mu)\cap\Pi=H^1(\mu)\cap L^2(\G).
\end{align}

\begin{lemma}\label{lem:hil}
The spaces $L^2(\G)$ and $H^1(\G)$ are Hilbert spaces, and $H^1(\G)$ is dense in $L^2(\G)$.
\end{lemma}

\begin{proof}
The pull-back $\pi_*$ is continuous since it is an isometry by definition, then $L^2(\G)$ and $H^1(\G)$ are closed  (relatively to $L^2(\mu)$ and $H^1(\mu)$). Therefore $L^2(\G)$ and $H^1(\G)$ are complete (for their respective norms induced by the norms on $L^2(\mu)$ and $H^1(\mu)$). 

To prove that $H^1(\G)$ is dense in $L^2(\G)$, we want to show that
\be
H^1(\G)^{\bot_{L^2(\G)}}=\acc{0}
\ee
where the orthogonal is taken in $L^2(\G)$. Then, by definition
\begin{align}\nonumber
H^1(\G)^{\bot_{L^2(\G)}}&=H^1(\G)^{\bot}\cap L^2(\G)=\overline{(H^1(\mu)^{\bot}\cup L^2(\G)^{\bot})}\cap L^2(\G)\\
&=L^2(\G)^{\bot}\cap L^2(\G)=\acc{0}.
\end{align}
We used that $H^1(\mu)$ (as a subspace of $L^2(\mu)$) is dense in $L^2(\mu)$ i.e. $H^1(\mu)^{\bot}=\acc{0}$.
\end{proof}

The limiting form is defined by projection (Equation \eqref{eq:intro.0}) but may, a priori, depend on $\al$. We denote $\cE_{\al}$ the projection of $E_{\al}$ and define its domain $\ccD(\cE_\al)$: for $u, v$ 
\begin{align}
\cE_{\al}(u,v)
&=E_{\al}(u\circ\pi,v\circ\pi)\\
\ccD(\cE_{\al})
&=\acc{u\in L^2(\G), E_\al^s(u\circ \pi,u\circ\pi)<+\infty}
=\pi_*^{-1}(H^1(\mu))=H^1(\G).
\end{align}

\begin{theorem}\label{th:DF}
The form $\cE_{\al}$ does not depend on $\al$ and defines a Dirichlet form on $L^2(\G)$ with domain $H^1(\G)$, denoted $\cE$. Moreover, $\cE$ is regular and has the local property. We decompose $\cE=\cE^s+\cE^a$ into its symmetric and antisymmetric parts, for all $u,v\in H^1(\G)$
\begin{align}\label{eq:DFprojsym}
\cE^s(u,v)
&=\e\sum_{i\in I}\int_{A_i}\abs{\grad H}^2\barr{\del_1u \del_1v}\dint \mu\\\nonumber
&\quad-\frac12\sum_{i\in I}\int_{A_i} \frac1h \grad\cdot(hF) \barr{uv}\dint \mu
-\frac12\sum_{O\in\cV_*}\int_{\pi^{-1}(O)}\frac1h \grad\cdot(hF) \barr{uv}\dint \mu\\\label{eq:DFprojantisym}
\cE^a(u,v)
&=\frac12\sum_{i\in I}\int_{A_i}F\cdot\grad H\cro{\barr {v\del_1 u}-\barr {u\del_1 v}}\dint\mu.
\end{align}
\end{theorem}

Using this theorem, we can associate to $\cE_{\al}$ a continuous Hunt process (\cite{ma.rockner92} Theorem 
IV 3.5, section IV 4. a), Theorem V 1.5). Let us denote by $(\Om, \cF, (Z_t)_t)$ this process.

\begin{proof}[Proof of Theorem \ref{th:DF}]
We compute the projection of the antisymmetric part of $\cE_\al$. Let us recall that for a function $u$ on $\G$, $\barr u=u\circ \pi$ denotes its pull-back on $\bbR^2$. 
For $u,v \in H^1(\G)$, we have, from Equation \eqref{eq:DFantisym.2} and Lemma \ref{lem:decomp}
\begin{align}
\int_{\bbR^2}A\grad H\cdot&\cro{\barr v\grad \barr u-\barr u\grad \barr v}\dint \mu
=\sum_{i\in I}\int_{A_i}A\grad H\cdot\cro{v(H,i)\grad u(H,i)-u(H,i)\grad v(H,i)}\dint \mu\\\nonumber
&+\sum_{O\in\cV_*}\int_{\pi^{-1}(O)}A\grad H\cdot\cro{v(H,i)\grad u(H,i)-u(H,i)\grad v(H,i)}\dint \mu
\end{align}
Then, from Lemma \ref{lem:decomp}, $\grad H=0$ on $\pi^{-1}(O)$ for $O\in\cV_*$. We obtain that
\begin{align}
\int_{\bbR^2}A\grad H\cdot\cro{\barr v\grad \barr u-\barr u\grad \barr v}\dint \mu
&=\sum_{i\in I}\int_{A_i}A\grad H\cdot\grad H\cro{\barr{v\del_1 u}-\barr{u\del_1 v}}\dint \mu=0.
\end{align}
We also have
\begin{align}
\int F\cdot \cro{\barr v\grad \barr u-\barr u\grad \barr v}\dint \mu
=\sum_{i\in I}\int_{A_i}F\cdot\grad H\cro{\barr {v\del_1 u}-\barr {u\del_1 v}}\dint \mu.
\end{align}
Therefore, we have
\begin{align}
\cE^a_{\al}(u,v)
=\frac12\sum_{i\in I}\int_{A_i}F\cdot\grad H\cro{\barr {v\del_1 u}-\barr {u\del_1 v}}\dint\mu.
\end{align}
Note that for the symmetric part $\cE^s_\al$ from \eqref{eq:DFsym.2}, the same calculation holds
\begin{align}
\cE^s_{\al}(u,v)
&=\e\sum_{i\in I}\int_{A_i}\abs{\grad H}^2\barr{\del_1u \del_1v}\dint \mu\\\nonumber
&\quad-\frac12\sum_{i\in I}\int_{A_i} \frac1h \grad\cdot(hF) \barr{uv}\dint \mu
-\frac12\sum_{O\in\cV_*}\int_{\pi^{-1}(O)}\frac1h \grad\cdot(hF) \barr{uv}\dint \mu.
\end{align}

Thus $\cE_{\al}$ is a bilinear form defined on $H^1(\G)$ and does not depend on $\al$ (thus denoted $\cE$). We have to prove that $(\cE, H^1(\G))$ is a Dirichlet form following the definition recalled before Assumption \ref{ass.2}. Let us denote $\cE^1$ and $\cE^{s,1}$ the forms $E^1_\al$ and $E^{s,1}_\al$ (defined by Equations \eqref{eq:N1} and \eqref{eq:N2}) projected on $H^1(\G)$.

Since $E^s_\al$ is positive definite on $H^1(\mu)$, $\cE^s$ is also positive definite on $H^1(\G)$. From Lemma \ref{lem:hil}, we know that $H^1(\G)$ is dense in $L^2(\G)$. $(\cE, H^1(\G))$ is closed on $L^2(\G)$ since $(\cE^{s,1})^{1/2}$ is a norm equivalent to $\norm{\cdot}_{H^1(\G)}$. The coercivity of $\cE$ is also inherited from the coercivity of $E_\al$. The contraction property is satisfied since it is satisfied by $E_{\al}$ and that, for $u\in H^1(\G)$
\be
\min(u_+,1)\circ \pi=\min((u\circ \pi)_+,1).
\ee

The local property is also trivially satisfied since $E_{\al}$ has the local property and if $u,v \in 
H^1(\G)$ are such that $\supp u\cap\supp v=\vide$, then
\be
\supp (u\circ \pi)\cap\supp (v\circ \pi)\subset \pi^{-1}(\supp u\cap\supp v)=\vide.
\ee

It now remains to show the regularity of the Dirichlet form.
Since $C_c\cap H^1(\mu)$ is dense in $H^1(\mu)$, we have $(C_c\cap H^1(\mu))^{\bot}=0$
where the orthogonal is taken with respect to the $H^1$-scalar product.
Then we have:
\begin{align}\nonumber
(C_c(\G)\cap H^1(\G))^{\bot_{H^1(\G)}}
&=(C_c\cap H^1(\mu)\cap H^1(\G))^{\bot}\cap H^1(\G)\\\nonumber
&=\overline{((C_c\cap H^1(\mu))^{\bot}\cup H^1(\G)^{\bot})}\cap H^1(\G)\\
&=H^1(\G)^{\bot}\cap H^1(\G)=0
\end{align}
which prove the density of $C_c(\G)\cap H^1(\G)$ in $H^1(\G)$.
The fact that $C_c(\G)\cap H^1(\G)$ is dense in $C_c(\G)$ (for the uniform norm) is a consequence of the Stone-Weierstrass theorem.
\end{proof}

\subsection{Mosco-convergence}

We would want to prove the averaging principle by the Mosco convergence of the Dirichlet form $(E_{\al}, H^1(\mu))$ on $L^2(\mu)$ to $(\cE, H^1(\G))$ defined on $L^2(\G)$ as $\al\to0$.
We then define (following \cite{hino98}) weak and strong convergence for elements in $L^2(\mu)$ to elements in $L^2(\G)$. 

\begin{definition}[Convergence]
The sequence $(u_n)$ in $L^2(\mu)$ strongly converge to $u\in L^2(\G)$, noted $u_n\conver{n}{+\infty}{L^2(\G)}u$, if 
\be
\norm{u_n-u\circ\pi}_{L^2(\mu)}\conv{n}{+\infty}0.
\ee
The sequence $(u_n)$ in $L^2(\mu)$ weakly converge to $u\in L^2(\G)$, noted $u_n\wconver{n}{+\infty}{L^2(\G)}u$, if 
\be
\left\{
\begin{aligned}
&\sup_n\norm{u_n}_{L^2(\mu)}<+\infty\\
&\bra{u_n, v\circ\pi}_{L^2(\mu)}\conv{n}{+\infty}\bra{u\circ\pi,v\circ\pi}_{L^2(\mu)}=\bra{u,v}_{L^2(\G)}, \forall v\in L^2(\G).
\end{aligned}
\right.
\ee
\end{definition}

We have the following proposition.
\begin{proposition}
We have the equivalence:
\be
u_n\conver{n}{+\infty}{L^2(\G)}u
\Longleftrightarrow
\left\{
\begin{aligned}
&u_n\wconver{n}{+\infty}{L^2(\G)}u\\
&\norm{u_n}_{L^2(\mu)}\conv{n}{+\infty}\norm{u\circ\pi}_{L^2(\mu)}.
\end{aligned}
\right.
\ee
\end{proposition}

\begin{proof}
This proposition is the transposition of the usual equivalence in $L^2(\mu)$.
\end{proof}

\begin{definition}[Mosco-convergence \cite{hino98}]\label{def:mosco}
For a sequence $(\al_n)$ converging to $0$, $E_{\al_n}$ Mos\-co-con\-ver\-ges to $\cE$ if the two following conditions hold:
\begin{enumerate}
\item for any sequence $(u_n)$ in $L^2(\mu)$ weakly converging to $u\in L^2(\G)$ and such that
\be
\sup_n E_{\al_n}^{s,1}(u_n)<+\infty
\ee
then $u$ is in $H^1(\G)$.
\item for all $v\in H^1(\G)$, and all sequences $(u_n)\in H^1(\mu)$ weakly converging in $L^2(\mu)$ to $u\in H^1(\G)$, there exists a sequence $(v_n)$ strongly converging to $v$, such that
\be
E_{\al_n}(u_n,v_n)\conv{n}{+\infty}\cE(u,v)
\ee
\end{enumerate}
The second condition could however be replaced by a weaker condition (see \cite{hino98} Section 3, Condition (F'2)): we ask that for all subsequences $(\al_{n_k})$, all sequences $(u_k)$ weakly convergent in $L^2(\mu)$ to $u\in H^1(\G)$ such that
\be
\sup_n E_{\al_{n_k}}^{s,1}(u_{k})<+\infty
\ee
and for all $v\in H^1(\G)$, there exists a sequence $(v_k)$ strongly converging to $v$ such that 
\be\label{eq:liminf}
\liminf_k E_{\al_{n_k}}(u_k,v_k)\ieg\cE(u,v)
\ee
\end{definition}

\begin{theorem}
For all sequence $(\al_n)$ converging to $0$, $E_{\al_n}$ Mosco-converges to $\cE$.
\end{theorem}

\begin{proof}
We have to check that the sequence $E_{\al_n}$ satisfies the condition of the definition \ref{def:mosco}.

Let us check the first condition. Let $(u_n)$ be a sequence in $L^2(\mu)$ weakly converging to $u$ in $L^2(\G)$ and such that $E^{s,1}_{\al_n}(u_n)$ is bounded. Then since $E^{s,1}_\al$ does not depend on $\al$ (Equation \eqref{eq:DFsym.2}) and is equivalent to the norm in $H^1$, the sequence $(u_n)$ is bounded in $H^1(\mu)$. Therefore there exists a subsequence converging weakly in $H^1(\mu)$, thus converging weakly in $L^2(\mu)$. By uniqueness of the limit, we deduce that $u\circ\pi$ is in $H^1(\mu)$.

For the second condition, we prove the alternate form (Equation \eqref{eq:liminf}). We fix a sequence $(u_n)$ weakly convergent in $L^2(\bbR^2)$ to $u\in H^1(\G)$ and bounded in $H^1(\bbR^2)$. Let us consider the constant sequence $v_n=v\circ\pi$, then the Condition \eqref{eq:liminf} becomes
\be
\liminf_k E_{\al_{n_k}}(u_k,v\circ\pi)\ieg\cE(u,v).
\ee
Since $(u_k)$ is bounded in $H^1(\mu)$, it is weakly precompact in $H^1(\mu)$. Let $w\in H^1(\mu)$ be a weak accumulation point of $(u_k)$ in $H^1(\mu)$, then it is also a weak accumulation point in $L^2(\mu)$, thus $u\circ\pi=w$. Therefore $u_k$ weakly converge in $H^1(\mu)$ to $u\circ\pi$. 

Since by Assumption \ref{ass.2}, $h^{-1}\grad\cdot(hF)$ is bounded, we have the convergence of the symmetric part \eqref{eq:DFsym.2}. The convergence of the antisymmetric part \eqref{eq:DFantisym.2} comes, for the term with $F$, from the fact that $F$ is assumed to be bounded (Assumption \ref{ass.2}). For the other term, with $A\grad H$, using \eqref{eq:int.2}, we have that
\begin{align}\nonumber
\frac1{2\al_{n_k}}\int A\grad H\cdot (u_{k}\grad \barr v-\barr v\grad u_{k})\dint \mu
&=\frac1{\al_{n_k}}\int (A\grad H\cdot \grad \barr v) u_{k}\dint \mu\\
&+\frac1{2\al_{n_k}}\int\frac1h\grad \cdot (hA\grad H)u_k\barr v\dint \mu.
\end{align}
The first term vanishes because $\grad\barr v=\grad H\barr{\del_1 v}$. The second vanishes since $\grad h\cdot A\grad H=0$ by Assumption \ref{ass.2}.
\end{proof}

\subsection{Convergence of finite dimensional marginals}

Let us denote $Y^{\al}$ the process defined by Equation \eqref{eq:diff.2}. It is associated to the Dirichlet forms $E_{\al}$.

An important consequence of the Mosco convergence of the Dirichlet form is the strong convergence of the semigroup, resolvent and generator associated to the form. 

\begin{definition}
Let $(B_n)_n$ be a sequence of bounded operators on $L^2(\mu)$ and $B$ a bounded operator on $L^2(\G)$, then:
\begin{itemize}
\item $(B_n)$ strongly converge to $B$ if for every sequence $(u_n)$ strongly converging to $u$, $(B_nu_n)$ strongly converge to $Bu$,
\item $(B_n)$ weakly converge to $B$ if for every sequence $(u_n)$ weakly converging to $u$, $(B_nu_n)$ weakly converge to $Bu$.
\end{itemize}
\end{definition}

The following Theorem (\cite{hino98}, Theorem 3.5, \cite{toelle06}, Theorem 2.53) gives us the convergence of the $C_0$-contraction semigroups and resolvants associated to $E_{\al_n}$ and $\cE$. 
\begin{theorem}
Let $T^n$, $T$ be the $C_0$-semigroups and $G^n_{\la}$, $G_{\la}$ be the $C_0$-resolvents associated to $E_{\al_n}$ and $\cE$. We have the following equivalence:
\begin{enumerate}
\item $T^n_t$ strongly converges to $T_t$ for all $t\seg0$;
\item $G^n_{\la}$ strongly converges to $G_{\la}$ for all $\la\seg0$;
\item $E_{\al_n}$ Mosco-converges to $\cE$.
\end{enumerate}
\end{theorem}

\begin{remark}
Let us consider $f_n$ strongly converging to $\barr f$, the strong convergence of $T^n_tf_n$ implies that 
\be\label{eq:convsg}
\norm{T^n_tf_n-\barr{T_tf}}_{L^2(\mu)}\conv{n}{+\infty}0.
\ee
\end{remark}

We now consider the convergence of finite dimensional distributions of the processes $Z^{\al_n}=(\pi(Y^{\al_n}))_n$. We suppose that the law of the initial condition $\nu_n=\ccL(Y^{\al_n}_0)$ has a density $g_n$ with respect to $\mu$ and $g_n$ converges weakly in $L^2(\mu)$ to $g\in L^2(\G)$ with $\bra{\barr g, \1}_{L^2(\mu)}=1$. Then $\nu_n$ converges weakly to the probability measure $\barr g\dint \mu$ which defines a probability measure $\wt\nu$ on $\G$. In particular, it means that for all $f_n\in C_{b}(\bbR^2)$ converging strongly uniformly to $\barr f$
\be
\int_{\bbR^2}f_n\dint \nu_n\conv{n}{+\infty}\int_{\bbR^2}\barr f \barr g\dint \mu=\int_{\G}f\dint \wt\nu.
\ee

We consider $Z$ the Markov process on $\G$ associated to the Dirichlet form $\cE$ with initial law $\nu$.

\begin{proposition}\label{prop:fdimcv}
We consider $(Y^{\al_n})$ the processes with initial distribution $\nu_n$ and the process $Z$ with initial distribution $\nu$. Under the condition that $\nu_n$ converges weakly to $\nu$, the finite dimensional distributions of $(Z^{\al_n})=(\pi(Y^{\al_n}))$ converge as $n$ goes to infinity to the finite dimensional distributions of $Z$.
\end{proposition}

By this proposition, we mean the following: for all $N>0$, $0<t_1<t_2<\cdots<t_N$, and $f_1,\cdots f_N \in  C_{b}(\G)$ then 
\be
\bbE_{\nu_n}[f_1(Z^{\al_n}_{t_1})f_2(Z^{\al_n}_{t_2})\cdots f_N(Z^{\al_n}_{t_N})]\conv{n}{+\infty}
\bbE_{\nu}[f_1(Z_{t_1})f_2(Z_{t_2})\cdots f_N(Z_{t_N})].
\ee

\begin{proof}
Let us prove the proposition for $N=2$, for the sake of simplicity.
We use the Markov property of the process $Y^{\al_n}$ for all $n$. Then 
\begin{align}\nonumber
\bbE_{\nu_n}[f_1(Z^{\al_n}_{t_1})f_2(Z^{\al_n}_{t_2})]
&=\bbE_{\nu_n}\cro{\barr{f_1}(Y^{\al_n}_{t_1})\barr{f_2}(Y^{\al_n}_{t_2})}
=\bbE_{\nu_n}\cro{\barr{f_1}(Y^{\al_n}_{t_1})\bbE_{Y^{\al_n}_{t_1}}[\barr{f_2}(Y^{\al_n}_{t_2})]}\\\nonumber
&=\bbE_{\nu_n}\cro{\barr{f_1}(Y^{\al_n}_{t_1})T^n_{t_2-t_1}\barr{f_2}(Y^{\al_n}_{t_1})}
=\int_{\bbR^2}T^n_{t_1}(\barr{f_1}T^n_{t_2-t_1}\barr{f_2})g_n\dint \mu\\\nonumber
&\conv{n}{+\infty}\int_{\bbR^2}\barr{T_{t_1}(f_1T_{t_2-t_1}f_2)}\barr{g}\dint \mu
=\int_{\G}T_{t_1}(f_1T_{t_2-t_1}f_2)\dint \nu\\
&\phantom{\conv{n}{+\infty}}=\bbE_{\nu}[f_1(Z_{t_1})f_2(Z_{t_2})].
\end{align}
The convergence comes from the fact that $g_n$ weakly converge to $\barr g$ and the fact that 
\be
T^n_{t_1}(\barr{f_1}T^n_{t_2-t_1}\barr{f_2})\conv{n}{+\infty}\barr{T_{t_1}(f_1T_{t_2-t_1}f_2)}
\ee
strongly in $L^2(\mu)$ by a repeated application of Equation \eqref{eq:convsg}.
\end{proof}

\subsection{Convergence in law}

The tightness of the law of $(Z^{\al_n})_{n}$  follows from the Lemma 3.2 of Chapter 8 of \cite{freidlin.wentzell12} for the case $e=0$. With Proposition \ref{prop:fdimcv} it gives the weak convergence of the law of the processes as $\al_n$ goes to $0$.

\begin{proposition}[Tightness, Freidlin Wentzell Lemma 3.2 Chapter 8]\label{prop:tight}
The family of distributions of $(Z^{\al_n})_{n}$ in the space $C(\bbR^+; \G)$ is tight.
\end{proposition}

The proof follows usual ideas to prove the tightness, and comes ultimately from Ascoli-Arz\'ela Theorem. The first argument is given by: for all $T>0$ and $\de>0$, there are $H_0$ and $n_0$, such that:
\be
n\seg n_0 \Longrightarrow \bbP_x[\max_{0\ieg t\ieg T}H(Y^{\al_n,\e}_t)\seg H_0]\ieg\de.
\ee

The equicontinuity condition comes from Stroock-Varadhan \cite{stroock.varadhan79}. Let us first precise the metric used on $\G$: the distance $\rho(y,y')$ is the minimum distance of the paths on $\G$. If $y=(H_0,i_0)$ and  $y'=(H'_0,i'_0)$, if $\g$ is a path from $y$ to $y'$  passing through the vertices $O_1, O_2\dots O_l$, we denote $\wt\rho$ the length of this path:
\be
\wt\rho(\g)=\abs{H_0-H(O_1)}+\sum_{i=1}^{l-1}\abs{H(O_i)-H(O_{i+1})}+\abs{H'_0-H(O_l)}.
\ee
Then $\rho(y,y')$ is simply the minimum of the lengths of all such paths.

The equicontinuity then follows from this: for all compact $K$ of $\G$ and all $\de$ sufficiently small there exists a constant $A_\de$ such that for every $a\in K$ there exists a function $f_{\de}^a$ with $f_\de^a(a)=1, f_\de^a(z)=0$ for $\rho(z,a)>\de$, and $0\ieg f^a_{\de}\ieg1$, such that, for all $n$, $f_{\de}^a(Z^{\al_n,\e}_t)+A_\de t$ is a submartingale. 

Note that their proof use the fact that $H$ has bounded second derivatives. Using our Assumption \ref{ass.1} on $e$, the generalization is straightforward.

We can then conclude by the convergence of the law of the processes $(Z^{\al_n})_n$ to the law of the process $Z$.
\begin{theorem}\label{th:conv}
For all $g_n$ weakly converging in $L^2$ to $g\circ\pi$, and such that $g_n$ and $g\circ\pi$ are densities (with respect to $\mu$) of probability laws in $\bbR^2$, then the processes $(\pi(Y^{\al_n,\e}))$ where $Y^{\al_n,\e}_0$ is distributed as $g_n\dint\mu$, converge in law to the process $Z^\e$ with initial law given by $g\dint(\pi_*\mu)$.
\end{theorem}

\begin{proof}
The theorem follows from Proposition \ref{prop:fdimcv} and Proposition \ref{prop:tight}.
\end{proof}

\section{Identification of the limiting process}\label{sec:iden}

We want to identify the limiting process on the space $\G$ by identifying the infinitesimal operator and its domain associated to the Dirichlet form $\cE$.
The domain of the operator $\cL$ associated to the Dirichlet form could be defined as (\cite{ma.rockner92} Proposition I 2.16 p.23):
\be\label{eq:domain}
\ccD(\cL)=\acc{u\in H^1(\G), v\mapsto\cE(u,v) \text{ is continuous w.r.t. $\norm{\cdot}_{L^2(\G)}$ on } H^1(\G)}.
\ee
Accordingly, the infinitesimal generator applied to a function $u\in \ccD(\cL)$ is given as $\cL u=w$ where $w\in L^2(\G)$ is the unique function such that, for all $v\in H^1(\mu)$
\be\label{eq:geninf}
\cE(u,v)=-\bra{w,v}_{L^2(\G)}.
\ee
In fact, we need to write the Dirichlet form as an integral with respect to the projected measure.

% Let us first prove a small lemma. We define
% \be
% R=\acc{x, \grad H(x)\neq 0}.
% \ee
% \begin{lemma}\label{lem:const}
% Assume that $h$ is a $C^1$ real-valued function such that $\grad h\cdot A\grad H=0$ identically, then there exists a function $\tht$ defined on $\G$ and of class $C^1$ on each edge, such that $h=\tht\circ\pi$ on the closure of the set $R$, denoted $\barr R$.
% \end{lemma}
% 
% \begin{proof}
% Let us consider the flow $\Phi_t(x_0)$ solution of the differential equation $\dot{x}=A\grad H(x)$. Let $x_0, x_1$ be on the same orbit of the flow: there exists $t_1$ such that $\Phi_{t_1}(x_0)=x_1$ then
% \be
% h(x_1)-h(x_0)=\int_0^{t_1}\frac{\dint}{\dt}h(\Phi_t(x_0))\dt=\int_0^{t_1}\grad h(\Phi_t(x_0))\cdot A\grad H(\Phi_t(x_0)))\dt=0.
% \ee
% Thus, $h$ is constant on the orbits and since $h$ is continuous, $h$ must be constant on the connected level sets of $H$.
% \end{proof}

%Following this Lemma, 
We denote $\tht$ the function defined on $\G$ such that $h=\tht\circ\pi$.
%  on $\barr R$. Note that at the vertices $O$, $\tht$ could not automatically be extended by continuity.
% We make the following assumption.
% \begin{assumption}\label{ass.3}
% We assume that $\tht$ could be extended by continuity on $\G$, i.e. the function $h$ is constant on the boundary of the sets $\pi^{-1}(O)$ for all vertices.
% \end{assumption}
% 
% \begin{remark}
% This condition implies that for all vertices $O$, and all edges $I_i$ incident to $O$, the limits
% \be
% \tht_i(O)=\lim_{m\to H(O)}\tht(m,i)
% \ee
% are the same for all $i$, denoted $\tht(O)$. Note that Lemma \ref{lem:const} already implies $\tht_i(O)=\tht_j(O)$ if the edges $I_i$ and $I_j$ are directly incident, i.e. if there is $x_0\in\barr R\cap\pi^{-1}(O)$ such that all neighborhood of $x$ intersects $A_i$ and $A_j$. In fact, since $h$ is continuous, we have
% \be
% \tht_i(O)=\lim_{x\to x_0,x\in A_i}h(x)=\lim_{x\to x_0,x\in A_j}h(x)=\tht_j(O).
% \ee
% Note that this assumption is verified if $h=\barr \tht$ on the whole space $\bbR^2$.
% \end{remark}

For all $i$, for all $m\in\mathring I_i$, we define 
\begin{align}
T(m,i)&=\oint_{C_i(m)}\frac{\dl}{\abs{\grad H}}&
\dint \om_{(m,i)}&=\frac1{T(m,i)}\frac{\dl}{\abs{\grad H}}.
\end{align}
Note that $T(m,i)$ is the period of the orbit of the Hamiltonian flow along the orbit $C_i(m)$ and $\om_{(m,i)}$ is the invariant measure of mass $1$ for this flow on $C_i(m)$, we would call it the normalized Liouville measure.
Let us also define
\begin{align}
S^2(m,i)
&=\oint_{C_i(m)}\abs{\grad H}^2\dint\om_{(m,i)}
=\frac1{T(m,i)}\oint_{C_i(m)}\abs{\grad H}\dl\\
B^0(m,i)
&=-\oint_{C_i(m)}e\cdot\grad H\dint \om_{(m,i)}
=-\frac1{T(m,i)}\oint_{C_i(m)}e\cdot\grad H\frac{\dl}{\abs{\grad H}}\\
B^1(m,i)
&=\oint_{C_i(m)}\D H\dint\om_{(m,i)}
=\frac1{T(m,i)}\oint_{C_i(m)}\D H\frac{\dl}{\abs{\grad H}}.
\end{align}

We define, for all vertices $O$ and all edges $I_i$ incident to $O$,
\begin{align}
\al_i(O)&=\lim_{m\to H(O), m\in I_i}\int_{C_i(m)}\abs{\grad H}\dl,&
\g(O)&=
-\int_{\pi^{-1}(O)}\grad\cdot e\dx.
\end{align}
Note that due to Assumption \ref{ass.2}, $\g(O)$ is positive. We consider $J(O)$ the set of edges incident to $O$.
We define the partition of $J(O)$ into two sets
\begin{align}
J_+(O)&=\acc{i|i\in J(O), m\seg H(O), \forall m\in I_i},\\
J_-(O)&=\acc{i|i\in J(O), m\ieg H(O), \forall m\in I_i}.
\end{align}

For a function $u$ on $\G$, we denote $u_i=u_{|I_i}$ its restriction on the edge $I_i$ (which is a real interval).
We have the main theorem of this section.
\begin{theorem}\label{th:gendom}
The domain $\ccD(\cL)\subset H^1(\G)$ is the set of real-valued functions $u$ on $\G$ such that:
\begin{enumerate}
\item $u$ is continuous on $\G$, for all $i$, $u_i$ is in $H^2(I_i)$,
\item for all $(m,i)$ in the interior of an edge (i.e. $m\in \mathring{I_i}$), the differential operator
\begin{align}\label{eq:opgraph}
\cL^iu_i(m)
=\e S^2_iu''_i
+(B^0_i+\e B^1_i)u'_i
\end{align}
defines a continuous function on $\mathring I_i$,
\item these functions $\cL_i u_i$ have a common limit at a common vertex $O$, denoted $\cL u(O)$,
\item at a vertex $O$, we have the relation
\begin{multline}\label{eq:gluecond}
\g(O)u(O)+\e\pare{\sum_{i\in J_+(O)}\al_i(O)D_iu(O)-\sum_{i\in J_-(O)}\al_i(O)D_iu(O)}\\=-\abs{\pi^{-1}(O)}\cL u(O),
\end{multline}
where $\abs{\pi^{-1}(O)}$ is the Lebesgue measure (area) of $\pi^{-1}(O)$
\item $\cL u$ is in $L^2(\G)$.
\end{enumerate}
\end{theorem}

\begin{remark}
This theorem shows that the underlying process is specified on two different domains : on the edges, we have a diffusion given by the restriction of the infinitesimal generator (Equation \eqref{eq:opgraph}); at each vertex, we have a gluing condition (Equation \eqref{eq:gluecond}), deduced from the domain, defining the properties of the vertex. 

\begin{itemize}
\item Let us notice that the process $Z$ does not depend on the choice of the measure $\mu$. This confirms that $\mu$ is only a convenient tool we use to define the Dirichlet form.
\item The coefficients of the diffusion on each edge are averages, with respect to the measure $\dint \om$, of, respectively, $\abs{\grad H}^2, -e\cdot\grad H$ and $\D H$ on each connected level set $C_i(m)$.
\item This Theorem also gives a way to construct the reference measure $\mu$. In fact on each edge, we have a one dimensional diffusion therefore a natural candidate for a invariant measure which has a density with respect to the Lebesgue measure. Due to the tree-like structure of $\G$, we obtain a invariant measure on $G$ by adding suitable constants to these densities on each edge, then the lift on $\bbR^2$ give us a suitable candidate for $\mu$.
\end{itemize}
\end{remark}

First we derive from the Dirichlet form, the infinitesimal generator and its domain. Then we deduce the diffusion process on the edges and its behavior when it reaches a vertex.

\subsection{Proof of Theorem \ref{th:gendom}}

In order to compute the different coefficients of the generator, we need a lemma.
\begin{lemma}\label{lem:derive}
For a vector field $G$ of class $C^1$ on $A_i$, we have
\be\label{eq:derive.1}
\frac{\dint}{\dm}\oint_{C_i(m)}G\cdot\frac{\grad H}{\abs{\grad H}}\dl
=\oint_{C_i(m)}\grad\cdot G\frac{\dl}{\abs{\grad H}}.
\ee
For a function $g$, of class $C^1$ on $A_i$, we have
\be\label{eq:derive.2}
\frac{\dint}{\dm}\oint_{C_i(m)}g\abs{\grad H}\dl
=\oint_{C_i(m)}\cro{\grad g\cdot\grad H+g\D H}\frac{\dl}{\abs{\grad H}}.
\ee
\end{lemma}

The second property \eqref{eq:derive.2} is given in \cite{freidlin.wentzell12} (Lemma 1.1 p.265). We formulate a short proof.
\begin{proof}
This lemma can be proved by using the co-area formula and the divergence theorem on a domain $D_i(m_0,m)=\acc{x\in C_i(m'), m'\in[(m_0,m)]}$. Let us suppose that $m_0<m$. 
First, let us remark that the unit vector $n=\frac{\grad H}{\abs{\grad H}}$
is a normal vector to the curve $C_i(m')$ which points towards the exterior of the domain along $C_i(m)$ and inwards along $C_i(m_0)$.
Therefore, using the divergence Theorem on the domain $D_i(m_0,m)$ for the integrand $\grad \cdot G$, we obtain:
\begin{align}\nonumber
\int_{D_i(m_0,m)}\grad\cdot G\dx
&=\int_{\del D_i(m_0,m)}G\cdot\dint n_{ext}\\
&=\oint_{C_i(m)}G\cdot\frac{\grad H}{\abs{\grad H}}\dl-\oint_{C_i(m_0)}G\cdot\frac{\grad H}{\abs{\grad H}}\dl.
\end{align}
Using the co-area formula on the same integral, we have:
\begin{align}
\int_{D_i(m_0,m)}\grad\cdot G\dx
&=\int_{m_0}^m\oint_{C_i(m')}\frac{\grad\cdot G}{\abs{\grad H}}\dl\dm'.
\end{align}
The same holds if $m<m_0$. Thus the integral $\oint_{C_i(m_0)}G\cdot\frac{\grad H}{\abs{\grad H}}\dl$ is differentiable at $m_0$ and we obtain the result.

For \eqref{eq:derive.2}, we remark that
\be
g\abs{\grad H}=g\grad H\cdot\frac{\grad H}{\abs{\grad H}}
\ee
and apply \eqref{eq:derive.1} to $G=g\grad H$.
We obtain the result since $\grad\cdot\cro{g\grad H}=\grad g\cdot\grad H+g\D H$.
\end{proof}

\begin{proof}[Proof of Theorem \ref{th:gendom}]
% Let us consider $u\in \ccD(\cL)$. 
% By the characterization \eqref{eq:domain} of the domain, $u$ defines a continuous (in $L^2(\G)$) linear form. Since $H^1(\G)$ is dense in $L^2(\G)$ (Lemma \ref{lem:hil}), this linear form has a unique extension on $L^2(\G)$. By the Riesz representation Theorem, there exists $w\in L^2(\G)$ such that, for all $v\in H^1(\G)$, we have $\cE(u,v)=-\bra{w,v}_{L^2(\G)}$.
First, we compute the generator for functions which are in $C^2(\G)$. Let us consider functions $u,v$ in $C^2(\G)$.
For the symmetric part, Equation \eqref{eq:DFprojsym}, we get
\begin{align}\nonumber
\cE^s(u,v)
&=\sum_{i\in I}\int_{A_i}\e\abs{\grad H}^2u'_i(H)v'_i(H)
-\frac1{2h} \grad\cdot(hF) u_i(H)v_i(H)\dint \mu\\
&\qquad-\frac12\sum_{O\in\cV_*}\int_{\pi^{-1}(O)}\frac1{h} \grad\cdot(hF) \barr u\barr{v}\dint \mu.
\end{align}
We use the coarea formula (since $\abs{\grad H}\neq 0$ on each $A_i$), to obtain
\begin{align}\nonumber
\cE^s(u,v)
&=\sum_{i\in I}\int_{I_i}\e u'_i(m)v'_i(m)\oint_{C_i(m)}\abs{\grad H}h\dl\\\nonumber
&\quad-u_i(m)v_i(m)\frac12\oint_{C_i(m)}\frac{\grad\cdot(hF)}{\abs{\grad H}}\dl \dm\\\nonumber
&\quad-\frac12\sum_{O\in\cV_*}u(O)v(O)\int_{\pi^{-1}(O)}\grad\cdot(hF)(x) \dx \\
&=\sum_{i\in I}\int_{I_i}\e a_iu'_iv'_i\dm
-\frac12\int_{I_i}c_iu_iv_i\dm
-\frac12\sum_{O\in\cV_*}\g(O) u(O)v(O).
\end{align}
We have denoted $a$ and $c$ the quantities
\begin{align}\label{eq:ac}
a(m,i)
&=\oint_{C_i(m)}\abs{\grad H}h\dl,&
c(m,i)
&=\oint_{C_i(m)}\frac{\grad\cdot(hF)}{\abs{\grad H}}\dl.
\end{align}

For the antisymmetric part, Equation \eqref{eq:DFprojantisym}, we get, using also the coarea formula
\begin{align}\label{eq:temp1}
\cE^a(u,v)
&=\frac12\sum_{i\in I}\int_{I_i}b_i\cro{v_iu'_i-u_iv'_i}\dm
\end{align}
where $b$ denotes the quantity
\begin{align}\label{eq:b}
b(m,i)
&=\oint_{C_i(m)}hF\cdot\grad H\frac{\dl}{\abs{\grad H}}.
\end{align}

Then, since we assume that $u\in C^2(\G)$, we have for all $i$ 
\begin{align}\label{eq:ipp.1}
\int_{I_i}a_iu_i'v_i'\dm
&=\lim_{m\to m^+_i}(a_iu'_iv_i)(m)-\lim_{m\to m^-_i}(a_iu'_iv_i)(m)-\int_{I_i}\pare{a_iu_i'}'v_i\dm
\end{align}
where we denote $I_i=\cro{m^+_i,m^-_i}$ (from Proposition \ref{prop:orbit}). Then, we get that
\be
\lim_{m\to m^+_i}a_i(m)=\tht(O)\al_i(O)
\ee
where $O$ is the vertex incident to $I_i$ at $m^i_+$. The same holds at $m^-_i$.

By summing the integrals \eqref{eq:ipp.1} over the edges of $\G$, we rewrite the first part as a sum over the vertices.
Since $v$ has a unique value at each vertex we get
\begin{align}\nonumber
\sum_i&\int_{I_i}a_iu_i'v_i'\dm
=-\sum_i\int_{I_i}\pare{a_iu_i'}'v_i\dm\\
&+\sum_{O\in\cV}v(O)\pare{\sum_{i\in J_+(O)}\tht(O)\al_i(O)D_iu(O)
-\sum_{i\in J_-(O)}\tht(O)\al_i(O)D_iu(O)}.
\end{align}
We do the same calculation for the integral \eqref{eq:temp1} and get
\begin{align}\label{eq:ipp.2}\nonumber
\sum_{i\in I}\int_{I_i}b_iu_iv'_i\dm
&=-\sum_i\int_{I_i}\pare{b_iu_i}'v_i\dm
\\&+\sum_{O\in\cV}v(O)u(O)\pare{\sum_{i\in J_+(O)}\bt_i(O)-\sum_{i\in J_-(O)}\bt_i(O)}
\end{align}
where $\bt_i(O)$ is defined as
\be
\bt_i(O)=\lim_{m\to H(O)}b(m,i).
\ee

Without loss of generality, we suppose for a moment that $H(O)=0$. Let us choose $\de_0>0$ such that, for all $\de<\de_0$, the connected domain, denoted $\Om_\de$, of $H^{-1}(]-\de,\de[)$ containing $\pi^{-1}(O)$, satisfies
\be
\grad H(x)\neq 0 \text{ for all $x\in \Om_\de\setminus \pi^{-1}(O)$}.
\ee
Let us apply the divergence formula on $\Om_\de$ to the vector field $hF$:
\begin{align}\label{eq:divergence2}
\int_{\Om_\de}\grad\cdot(hF)\dx
&=\int_{\del\Om_\de}hF\cdot\dint n_{ext}.
\end{align}
Using the same method as in the proof of Lemma \ref{lem:derive}, we have
\begin{align}\nonumber
\int_{\Om_\de}\grad\cdot(hF)\dx
&=\sum_{i\in J_+(O)}\oint_{C_i(\de)}hF\cdot\frac{\grad H}{\abs{\grad H}}\dl-\sum_{i\in J_-(O)}\oint_{C_i(-\de)}hF\cdot\frac{\grad H}{\abs{\grad H}}\dl\\\label{eq:temp.2}
&=\sum_{i\in J_+(O)}b(\de,i)-\sum_{i\in J_-(O)}b(-\de,i).
\end{align}
Since $H$ is $C^1$, we get that 
\be
\int_{\Om_\de}\grad\cdot(hF)\dx\conv{\de}{0}\int_{\pi^{-1}(O)}\grad\cdot(hF)\dx=-\tht(O)\g(O).
\ee
We deduce from \eqref{eq:temp.2}
\begin{align}
\tht(O)\g(O)=\sum_{i\in J_+(O)}\bt_i(O)-\sum_{i\in J_-(O)}\bt_i(O).
\end{align}

Therefore, after simplification, the whole Dirichlet form is
\begin{align}\label{eq:DFproj}
\cE(u,v)
&=-\sum_i\int_{I_i}v_i\cro{\e a_iu''_i
-(b_i-\e a'_i)u'_i+\frac12(c_i-b'_i)u_i}\dm\\\nonumber
&\quad+\sum_{O\in\cV}\tht(O)v(O)\Bigg[\e\bigg(\sum_{i\in J_+(O)}\al_i(O)D_iu(O)-\sum_{i\in J_-(O)}\al_i(O)D_iu(O)\bigg)\\\nonumber
&\quad+u(O)\g(O)\Bigg].
\end{align}

We need to identify the projected measure  $\dint\pi_*\mu$: for $v\in C_c(\G)$, by definition, $\barr v=v\circ\pi$ is continuous with compact support, thus integrable. We get, using also the coarea formula
\begin{align}\nonumber\label{eq:projmeas}
\int_{\G}v\dint \pi_*\mu
&=\int_{\bbR^2}\barr v\dint \mu
=\sum_i\int_{A_i}\barr v h\dx+\sum_{O\in\cV_*}\int_{\pi^{-1}(O)}\barr v h\dx
\\\nonumber
&=\sum_i\int_{I_i}v_i(m)\oint_{C_i(m)}\frac{h}{\abs{\grad H}}\dl\dm
+\sum_{O\in\cV_*}v(O)\mu(\pi^{-1}(O))\\
&=\sum_i\int_{I_i}v_id_i\dm
+\sum_{O\in\cV_*}v(O)\tht(O)\abs{\pi^{-1}(O)}
\end{align}
where $d$ denotes
\be
d(m,i)
=\oint_{C_i(m)}\frac{h}{\abs{\grad H}}\dl=\tht_i(m)T_i(m).
\ee

To identify the operator $\cL$, we solve \eqref{eq:geninf} for any $v\in C^2(\G)$.
Using Equations \eqref{eq:DFproj} and \eqref{eq:projmeas}, which are valid for any test functions $v$, therefore we see that $w$ must satisfy
\begin{align}\label{eq:generator1}
w_id_i
&=\e a_iu''_i
-(b_i-\e a'_i)u'_i+\frac12(c_i-b'_i)u_i\\\label{eq:gluing1}
w(O)\tht(O)\abs{\pi^{-1}(O)}
&=\tht(O)\g(O)u(O)\\\nonumber
&\quad-\e\tht(O)\pare{\sum_{i\in J_+(O)}\al_i(O)D_iu(O)-\sum_{i\in J_-(O)}\al_i(O)D_iu(O)}.
\end{align}

Thus by definition, the domain of the operator is the set of functions $u$ such that, on each edge $I_i$, we can define $w_i=\cL_iu_i$, using Equation \eqref{eq:generator1} 

We deduce that $u\in H^2(I_i)$ for every edge $I_i$, therefore $u$ has at the boundary of each edge limits of the first derivatives.
The value of $w$ on each vertex is given by continuity. Therefore, these limits must coincide and satisfy Condition \eqref{eq:gluing1}. Then we can define $w=\cL u$.
The restriction is that $w$ must be in $L^2(\G)$
\begin{align}
\ccD(\cL)=\acc{u\in H^1(\G), \cL u\in L^2(\G)}.
\end{align}

% Also, we see that $\barr u$ must be in the domain of the original infinitesimal generator which is included in $H^2(\mu)$, therefore $\barr u$ must be continuous and $u$ is also continuous.

To finish the proof we need to compute the coefficients of the generator in Equation \eqref{eq:generator1}.
Using Equation \eqref{eq:derive.2} for $a$ (given by \eqref{eq:ac}) and Equation \eqref{eq:derive.1} for $b$ (given by \eqref{eq:b}) from Lemma \ref{lem:derive}, we have
\begin{align}
a'_i(m)
&=\oint_{C_i(m)}\cro{\grad h\cdot\grad H+h\D H}\frac{\dl}{\abs{\grad H}},&
b'_i(m)
&=\oint_{C_i(m)}\grad\cdot(hF)\frac{\dl}{\abs{\grad H}}.
\end{align}
Then we see that $b'_i=c_i$ (given by Equation \eqref{eq:ac}) on each edge $I_i$.

Since $h=\barr{\tht}$, the vector field $F$ is
\be
F=e+\e\frac{\barr{\del_1\tht}}{\barr{\tht}}\grad H.
\ee
Therefore, we have, from Equations \eqref{eq:ac} and \eqref{eq:b}
\begin{align}
a_i(m)
&=\tht_i(m)\oint_{C_i(m)}\abs{\grad H}\dl\\
a'_i(m)
&=\tht_i(m)\oint_{C_i(m)}\D H\frac{\dl}{\abs{\grad H}}+\tht'_i(m)\oint_{C_i(m)}\abs{\grad H}\dl\\
b_i(m)
&=\tht_i(m)\oint_{C_i(m)}e\cdot\grad H\frac{\dl}{\abs{\grad H}}
+\e\tht'_i(m)\oint_{C_i(m)}\abs{\grad H}\dl.
\end{align}

The coefficients of the generator (Equation \eqref{eq:generator1}) could then be written
\begin{align}
\frac{a_i(m)}{d_i(m)}
&=\frac1{T_i(m)}\oint_{C_i(m)}\abs{\grad H}\dl=S_i^2(m)\\
-\frac{1}{d_i(m)}\pare{b_i(m)-\e a'_i(m)}
&=\frac1{T_i(m)}\oint_{C_i(m)}\cro{\e\D H-e\cdot\grad H}\frac{\dl}{\abs{\grad H}}\\\nonumber
&=\e B^1_i(m)+B^0_i(m)
\end{align}
which give us the result.
\end{proof}

\subsection{Local behavior}

Theorem \ref{th:gendom} is sufficient to define and describe the process $Z^\e$. However we would like to give a more intuitive description of the process. %We have the following properties.
% \begin{lemma}\label{lem:loc}
Note that for any vertex $O$, we have:
\begin{enumerate}
\item if $\abs{\pi^{-1}(O)}=0$ then $\g(O)=0$;
\item the coefficients $\al_i(O)$ are such that
\be\label{eq:flux}
\sum_{i\in J_+(O)}\al_i(O)=\sum_{i\in J_-(O)}\al_i(O).
\ee
\end{enumerate}

From Theorem \ref{th:gendom}, we see that on each edge, the process is a continuous diffusion whose characteristics are explicitly given as averaging along the connected level sets of the Hamiltonian $H$. However at the edges, the gluing conditions are not so clear could give several different behavior. 

Note that at the vertices two issues must be addressed in order to successfully describe the behavior of the process :
\begin{itemize}
\item is the vertex accessible and from which edges ?
\item what happend when the process reach the vertex ?
\end{itemize}

We can determine the different gluing conditions we can have. 
\begin{itemize}
\item For an exterior vertex $O$ (e.g. a vertex with only one incident edge), we get two types of boundary conditions
\begin{enumerate}
\item no gluing conditions if $\pi^{-1}(O)$ is a null measure set (i.e. a single point),
\item or the boundary condition 
\be
\abs{\pi^{-1}(O)}\cL u(0)=\g(O)u(O).
\ee
\end{enumerate}
\item For an interior vertex $O$, we could have again two types of boundary conditions (with the relation \eqref{eq:flux})
\begin{enumerate}
\item purely first order gluing condition if $\pi^{-1}(O)$ is a null measure set
\be
\sum_{i\in J_+(O)}\al_i(O)D_iu(O)-\sum_{i\in J_-(O)}\al_i(O)D_iu(O)=0,
\ee
\item mixed gluing conditions
\be
\abs{\pi^{-1}(O)}\cL u(0)=\g(O)u(O)-\sum_{i\in J_+(O)}\al_i(O)D_iu(O)+\sum_{i\in J_-(O)}\al_i(O)D_iu(O).
\ee
\end{enumerate}
\end{itemize}

Detailed analysis of such process has been conducted by several authors as Feller \cite{feller54} or Mandl \cite{mandl68}. We also refer to \cite{KKVW09} and \cite{KPS12}.

\section{Generalization}\label{sec:general}

In this last section, we present a generalization of our previous results for more general diffusions. However, we only sketch the computation of the generator of the diffusion from the limiting Dirichlet form.

Consider the diffusion defined on $\bbR^n$ by
\be\label{eq:diff.3}
\dint Y_t=\frac1{\al}v(Y_t)\dt+u(Y_t)\dt+\sqrt{2\e}\s(Y_t)\dint B_t.
\ee
The vector field $v$ plays the role of the $A\grad H$ for the $2$-dimensional case and $u$ plays the role of a friction term. Since we do not suppose that $v$ is given by some Hamiltonian, we assume instead existence of some $m$ first integrals $G=(G_1,\cdots,G_m)$ for the flow defined by $v$. 
We wish to derive the convergence in law of the process $G(Y)$ as $\al$ goes to $0$.

In this part, for a vector $Y$ or a matrix $a$, we denote by $Y^*$ or $a^*$ the transposition of these elements. In particular a scalar product between two vectors $X$ and $Y$ could be written
\be
X^*Y=X\cdot Y=\sum_iX_iY_i
\ee
where the usual matrix product takes place in the left-hand side expression. 

Let us denote $a=\s\s^*$ the diffusion matrix which is a symmetric matrix.
For the function $G: \bbR^n\to\bbR^m$, we define the matrix $DG$ by
\be
DG=\pare{
\begin{matrix}
\del_1G_1&\cdots&\del_nG_1\\
\vdots&&\vdots\\
\del_1G_m&\cdots&\del_nG_m
\end{matrix}}.
\ee

We make the following assumptions.
\begin{assumption}\label{ass.4}
We assume that $v, u$ and $\s$ are Lipschitz bounded functions and that $a$ is uniformly elliptic.. The function $G$ satisfies
\be\label{eq:vpreserved}
DG v=0.
\ee
We suppose also that $G$ has compact level sets.
We assume that there exists a $C^1$ function $h$ strictly positive, which satisfies
\begin{align}\label{eq:mupreserved}
\grad ^*[vh]&=0\\\label{eq:supermedian.2}
-\grad^*[uh]+\e\sum_{i,j}\del_{ij}(a_{ij}h)&\ieg 0.
\end{align}
\end{assumption}
We define the measure $\dint \mu=h\dx$.

\begin{remark}
The assumptions on $v$, $u$ and $\s$ ensure the existence of a strong solution to the stochastic differential equation \eqref{eq:diff.2}. Equation \eqref{eq:vpreserved} ensures that $G$ is conserved along orbits of the flow of $v$. The function $h$ is, as in the first part, the density of a measure $\mu$. Equation \eqref{eq:mupreserved} ensures that $\mu$ is also preserved by the flow generated by $v$. To define a proper Dirichlet form, $\mu$ must be supermedian (for the infinitesimal generator of the diffusion \eqref{eq:diff.3} see \cite{ma.rockner92} pp.62 and 98), Equation \eqref{eq:supermedian.2} ensures that.
\end{remark}

\subsection{Dirichlet form}

This first proposition gives the Dirichlet form for the diffusion \eqref{eq:diff.3}.

Let us denote the vector fields $\Phi, \Phi_{\al}$:
\begin{align}
\Phi&=u-\frac{\e}{h}(\grad^*(ah))^*\\
\Phi_{\al}&=\frac1{\al}v+\Phi=\frac1{\al}v+u-\frac{\e}{h}(\grad^*(ah))^*.
\end{align}
The infinitesimal generator, $L_{\al}$, could be written, for $f\in C^{2}_{c}(\bbR^n)$
\begin{align}
L_{\al}f&=\frac1{\al}v^* \grad f+u^* \grad f+\e\sum_{i,j}a_{ij}\del_{ij}f
=\frac1{\al}v^*\grad f+u^*\grad f+\e a:\grad^2 f
\end{align}
where $a:b$ denotes the Frobenius product of the two matrices and $\grad^2 f$ is the matrix of the second derivatives of $f$.

We consider the Dirichlet form $E_{\al}$ associated to $L_{\al}$ in $L^2(\mu)$. For $f,g\in C^2_c$, we have
\begin{align}\label{eq:dir}
E_{\al}(f,g)&=-\bra{L_{\al}f,g}_{\mu}
=-\frac1{\al}\int (v^*\grad f)g\dint\mu-\int (u^*\grad f)g\dint\mu -\e\int (a:\grad^2 f)g\dint\mu.
\end{align}

\begin{proposition}
Under the set of assumptions \ref{ass.4}, $E_{\al}$ is coercive, closable and can be extended to a Dirichlet form $E_{\al}$ on $L^2(\mu)$ with domain $\ccD(E_\al)=H^1(\mu)$ defined as the completion in $L^2(\mu)$ of $\ccD(L_\al)$.
Moreover $E_\al$ is regular and has the local property. The Dirichlet form is then $E_{\al}(f,g)=E^s_{\al}(f,g)+E^a_{\al}(f,g)$
\begin{align}\label{eq:DFsym.3}
E^s_{\al}(f,g)
&=\e\bra{\s^*\grad f,\s^*\grad g}_{\mu}+\frac12\int \frac1{h}\grad^*(\Phi h)fg\dint \mu\\\label{eq:DFantisym.3}
E^a_{\al}(f,g)
&=\frac12\int \Phi_{\al}^*\cro{f\grad g-g\grad f}\dint \mu
=\frac12\int \pare{\frac1{\al}v+\Phi}^*\cro{f\grad g-g\grad f}\dint \mu.
\end{align}
\end{proposition}

\begin{proof}
Let us first consider $f,g\in C_c^2(\bbR^n)$.
The first and second integrals in Equation \eqref{eq:dir} are treated as in \eqref{eq:int.2}. For the third integral, we get
\begin{align}\label{eq:diffusionpart}\nonumber
-\int (a:\grad^2 f)g\dint \mu&=-\sum_{ij}\int (a_{ij}\del_{ij}f)gh
=\sum_{ij}\int \del_{j}f\del_i(a_{ij}gh)\\\nonumber
&=\sum_{ij}\int a_{ij}(\del_{j}f\del_ig)h+\sum_{ij}\int \del_i(h)a_{ij}\del_{j}fg+\sum_{ij}\int \del_{j}f\del_i(a_{ij})gh\\
&=\int (\s^*\grad f)^*(\s^*\grad g)h+\int\cro{(a\grad h)^*\grad f} g+\int \cro{(\grad^* a)\grad f}gh.
\end{align}
Then we can decompose this bilinear form in $(f,g)$ in a symmetric and antisymmetric part using similar computations as in \eqref{eq:int.2}. Thus, Equation \eqref{eq:diffusionpart} becomes
\begin{align}\nonumber
-\int (a:&\grad^2 f)gh=
\bra{\s^*\grad f,\s^*\grad g}_{\mu}+\frac12\int(a\grad h+(\grad^* a)^*h)^*\cro{g\grad f- f\grad g}\\\nonumber
&\qquad-\frac12\int\cro{\grad^*(a\grad h)+h\grad^*((\grad^* a)^*)+(\grad h)^*(\grad^*a)^*}fg\\
&=\bra{\s^*\grad f,\s^*\grad g}_{\mu}+\frac12\int\cro{\grad^* (ah)}\cro{g\grad f- f\grad g}-\frac12\int\cro{\grad^2:(a h)}fg.
\end{align}

Putting all the terms together we obtain the announced Dirichlet form.

The domain $\ccD(E_\al)$ is defined as the completion of $\ccD(L_\al)$ with respect to the positive bilinear form $E_\al^{s,1}$ and it defines a coercive closed bilinear form (\cite{ma.rockner92} Theorem 2.15 p.22). The contraction properties \eqref{eq:contractions} are proved in \cite{ma.rockner92} Section II.2.d. (p.48) under suitable conditions (Equations (2.16) p.48) which come as a consequence of our assumption \eqref{eq:supermedian.2}.

The fact that the domain is $H^1(\mu)$ comes from the fact that, under the assumption \ref{ass.4}, $\Phi_\al$ and $\grad\cdot(h\Phi)$ are bounded and $a$ is uniformly elliptic. Therefore $E_\al^{s,1}$ and $\norm{\cdot}_{H^1(\mu)}$ are equivalent norms. The regularity is obvious from the properties of the sets $H^1(\mu)$ and $C_c$. The local property is also obvious from Equations \eqref{eq:DFsym.3} and \eqref{eq:DFantisym.3}.
\end{proof}

\subsection{Projected Dirichlet form}

We consider the projection $\pi$ associated to the equivalent classes defined by
\be\nonumber
x\sim y \Longleftrightarrow \text{ and $x,y$ are in the same connected component of a level set of $G$}.
\ee
The space $\G=\bbR^n/\sim$ is the new state space. However, a complete description of this state space is quite complicated and not really necessary for the Mosco-convergence.
As in the first part, for all $f$ defined on $\G$, we define $\barr{f}=f\circ \pi$, its lift on $\bbR^n$.

In order to have a Mosco-convergence of the Dirichlet form $E_\al$ as $\al\to0$, we have to make the terms in Equation \eqref{eq:DFantisym.3} depending on $\al$ vanishing, i.e. we want
\begin{align}
v^*\cro{f\grad g-g\grad f}&=0.
\end{align}
This is achieved by the choice of test functions and is equivalent to a projection of the Dirichlet form.
\begin{lemma}
For all $f\in C^1(\G)$, we have
\be
v^*\grad\barr{f}=0.
\ee
\end{lemma}

\begin{proof}
Let us consider the flow $\phi_t$ generated by the vector field $v$. We define, for all $x,t$
\be
F(x,t)=\barr{f}(\phi_t(x)).
\ee
By definition, we have $G(\phi_t(x))=G(x)$ for all $t$, therefore $\phi_t(x)$ stays in a unique equivalence class for all $t\in\bbR$. Then $F(x,t)$ is a constant function in $t$, thus
\be
\del_t F(x,t)=v^*\grad\barr{f}(\phi_t(x))=0, \text{ for all $t$}.
\ee
Putting $t=0$, we get the lemma.
\end{proof}

We test the Dirichlet form $E_{\al}$ on $\barr{f_1}=f_1\circ \pi, \barr{f_2}=f_2\circ \pi$. Using the previous lemma, we see that $E_{\al}(\barr{f_1},\barr{f_2})$ defines a bilinear form, for functions in $C^1_c(\G)$ which does not depend on $\al$.
We denote it $\cE$:
\begin{align}
\cE(f_1,f_2)
&=\e\int_{\bbR^n}(\grad \barr{f_1})^* a\grad \barr{f_2}\dint \mu
\\\nonumber
&\quad+\frac12\int \Phi^*\cro{\barr{f_1}\grad \barr{f_2}-\barr{f_2}\grad \barr{f_1}}\dint \mu
+\frac12\int \frac1{h}\grad^*(\Phi h)\barr{f_1}\barr{f_2}\dint \mu.
\end{align}

We have the following Proposition.
\begin{proposition}\label{prop:proDF}
The form $\cE$ define a Dirichlet form on $L^2(\G)$ with domain $H^1(\G)$. Moreover, $\cE$ is regular and has the local property. 
\end{proposition}

\begin{remark}
This proposition allows us to define a process $Z$ on $\G$ associated to the Dirichlet form, $\cE$.
\end{remark}

\begin{proof}
The proof follows exactly the proof of Theorem \ref{th:DF}.
\end{proof}

% \begin{remark}
% Note that $\mu$ is invariant for the diffusion
% \be\label{eq:diff}
% \dint X_t=u(X_t)\dt+\sqrt{2\e}\s(X_t)\dint B_t,
% \ee
% is equivalent to the condition that the vector field $\Phi$ satisfies
% $\grad^*(\Phi h)=0$.
% In this case, the projected Dirichlet form takes the form
% \begin{align}
% \cE(f_1,f_2)
% &=\e\int_{\bbR^n}(\grad \barr{f_1})^* a\grad \barr{f_2}\dint \mu
% +\frac12\int \Phi^*\cro{\barr{f_1}\grad \barr{f_2}-\barr{f_2}\grad \barr{f_1}}\dint \mu.
% \end{align}

% Moreover, note that $\mu$ is reversible for the diffusion \eqref{eq:diff}
% if and only if $\Phi=0$.
% In this case, the projected Dirichlet form is symmetric and could be written
% \begin{align}
% \cE(f_1,f_2)
% &=\e\int_{\bbR^n}(\grad \barr{f_1})^* a\grad \barr{f_2}\dint \mu.
% \end{align}
% \end{remark}

\subsection{Mosco-convergence and convergence in law}

Let us denote $Y^{\al}$ the process solution of the diffusion equation \eqref{eq:diff.3} with $\al>0$. We denote $Z$ the process associated to the Dirichlet form $\cE$.

We have the following Theorem.
\begin{theorem}\label{th:conv.2}
For all sequence $\al_n$ converging to $0$, for all $g_n$ weakly converging in $L^2$ to $g\circ\pi$, and such that $g_n$ and $g\circ\pi$ are densities (with respect to $\mu$) of probability laws in $\bbR^n$, the sequence of processes $Z^{\al_n}=\pi(Y^{\al_n})$ where $Y^{\al_n}_0$ is distributed as $g_n\dint\mu$, converges in law to the process $Z$ with initial law given by $g\dint(\pi_*\mu)$.
\end{theorem}

The proof follows exactly as in the first part (Theorem \ref{th:conv}) by proving two facts: the convergence of the finite dimensional marginals of the processes $(Z^{\al_n})$ (which follows itself form the convergence of the associated Dirichlet forms), and the tightness.

\begin{proposition}\label{prop:Mconv}
For all sequence $(\al_n)$ converging to $0$, $E_{\al_n}$ Mosco-converges to $\cE$.
\end{proposition}

\begin{proof}
The proof follows exactly the proof of Theorem \ref{th:DF}.
\end{proof}

We consider $Z$ the Markov process on $\G$ associated to the Dirichlet form $\cE$ with initial law $\nu$.

\begin{proposition}[Tightness]\label{prop:tight.2}
The family of distributions of $(Z^{\al_n})_{n}$ is tight in the space $C(\bbR^+; \G)$.
\end{proposition}

\begin{proof}
By the same proof of Proposition \ref{prop:tight}, we prove that each family of processes $G_i(Y^{\al_n})$ is tight, thus, using Corollary 3.33 p.317 in \cite{jacod.shiryaev} the family of processes $G(Y^{\al_n})$ is tight. Therefore $Z^{\al_n}$ is also tight since all bounded continuous function on $\G$ could be decomposed as a countable sum of bounded continuous functions depending only on $G$.
\end{proof}

\begin{proof}[Proof of Theorem \ref{th:conv.2}]
From Proposition \ref{prop:Mconv}, we deduce as in the proof of Proposition \ref{prop:fdimcv} the convergence of the finite dimensional marginals.
The theorem follows from the tightness proved in Proposition \ref{prop:tight.2}.
\end{proof}

\subsection{Identification of the limiting process}

The process $Z$ is only defined through the Dirichlet form $\cE$. In order to obtain a more intuitive representation of this process, through a stochastic diffusion equation for example or its infinitesimal generator, $\cL_\e$, one should write the Dirichlet form as a scalar product in $L^2(\G)$. In this last section, we would like to expose what one should expect and how computations could be made.

The equation is therefore the following, we wish to find $w\in L^2(\G)$ such that, for all $v\in H^1(\G)$
\begin{align}\label{eq:DF.25}\nonumber
\cE(u,v)&=\e\int_{\bbR^n}(\grad \barr{u})^* a\grad \barr{v}\dint \mu
+\frac12\int \Phi^*\cro{\barr{u}\grad \barr{v}-\barr{v}\grad \barr{u}}\dint \mu
+\frac12\int \frac1{h}\grad^*(\Phi h)\barr{u}\barr{v}\dint \mu\\
&=-\int_{\G}wv\dint \pi_*\mu.
\end{align}
In order to do this identification, one has to do several things:
\begin{itemize}
\item write the Dirichlet form $\cE$ as an integral on $\G$;
\item write the image measure of $\mu$ via the projection $\pi$ on $\G$.
\end{itemize}
%However, first of all, we need to identify more precisely the state space $\G$.

\subsubsection{Orbit Space}
As in the two dimensional case, $\G$ is decomposed into several connected submanifolds of $\bbR^m$. This decomposition is done by using the minimal rank of the Jacobian $JG$ along a connected level set of $G$. This rank is the dimension of the submanifold. Each submanifold of dimension higher than one has a boundary made of a union of submanifolds of dimension strictly lower. We describe $\G$ by this collection of manifolds ordered along their dimensions and by the relations describing their boundaries.

We denote $I^k_i$ the collection of submanifolds of dimension $k$ and $R^k_i$ the preimage associated to the submanifold $I^k_i$.

\subsubsection{Image measure}

On each submanifold, we need to do a change of variable to transfer an integral on the preimage of a submanifold $R^k_i$ (which is a domain in $\bbR^n$) to a integral on this submanifold $I^k_i$ (a domain in $\bbR^m$).

We denote $\dint\cH_k$ for $k\seg 0$, the $k$-dimensional Hausdorff measure.
% We express the integral of a function $f$ on $\bbR^n$ as an integral on $\G$, i.e. on each subset $I_i^k$ of $\G$. In fact since the measure of $R_0$ is null, we will have no contributions from the sets $R_i^k$ for $i<0$. Therefore for $f$ an integrable function on $\bbR^n$, we have
% \begin{align}
% \int_{\bbR^n}f(x)\dx
% =\int_{R_+}f(x)\dx+\int_{R_0}f(x)\dx
% =\int_{R_+}f(x)\dx
% =\sum_k\sum_{i\seg0}\int_{R^k_i}f(x)\dx.
% \end{align}

% We introduce the Jacobian $JG$ of the function $G$ defined as
% \be
% JG(x)=\sqrt{\det DG(x)DG(x)^*}.
% \ee
% Note that $JG(x)=0$ means that $\rk DG(x)<m$.
% On the set $R^m$ we can apply the usual coarea formula 
% (\cite{evans.gariepy92} Theorem 2 p.117), 
% \begin{align}\label{eq:coarea}
% \int_{R^m} f(x)\dx
% &=\sum_{i\seg0} \int_{R_i^m} f(x)\dx
% =\sum_{i\seg0}\int_{I^m_i}\int_{G^{-1}(g)\cap R^m_i}\frac{f}{JG}\dint \cH_{n-m}\dint g
% \end{align}
% where $\dint\cH_{n-m}$ is the surface measure on the level sets of $G$ which are of dimension $n-m$.

For $k=m$, this change of variable is just the coarea formula.
The integrals on subsets of rank strictly lower than $m$ must be treated carefully. In these cases, we need a coarea formula from a domain $R^k_i$ to $I_i^k$, a submanifold in $\bbR^m$ of dimension $k$. 
%This could be done by using a $C^1$-parameterization (i.e. a $C^1$ diffeomorphism) $G_1$ of the submanifold $I_i^k$ on a domain  $\Om_i^k\subset\bbR^k$, and by writing $G=G_1\circ G_2$, where $G_2$ is a $C^1$ function mapping $R^k_i$ to $\Om_i^k$. 

%Then using the coarea formula to $G_2$ and the area formula to $G_1$ (see e.g. \cite{evans.gariepy92}) 
We introduce a $k$-dimensional Jacobian $J_kG$ (\cite{krantz.parks08})
\be
J_kG(y)=\sup\acc{\frac{\cH_k(DG(y)P)}{\cH_k(P)}, \text{ for $P$ a $k$-dimensional parallelepiped in $\bbR^n$}}.
\ee
Remark that for the usual Jacobian we have $JG=J_mG$.
We obtain
\begin{align}
\int_{R^k_i}f(x)\dx
% &=\int_{\Om_i^k}\dx\int_{G_2^{-1}(x)\cap R^k_i}\frac{f(y)}{JG_2}\dy
=\int_{I_i^k}\dint \cH_k(g)\int_{G^{-1}(g)\cap R_i^k}\frac{f(y)}{J_kG(y)}\dint \cH_{n-k}(y).
\end{align}
%where $J_kG(y)=JG_1(G_2(y))JG_2$. 

% The whole integral could be decomposed as follows
% \begin{align}\label{eq:temp.213}
% \int_{\bbR^n} f(x)\dx
% &=\sum_{k\ieg m}\sum_{i\seg0}\int_{R_i^k}f(x)\dx
% =\sum_{k\ieg m}\sum_{i\seg0}\int_{I_i^k}\int_{G^{-1}(g)\cap R^{k}_i}\frac{f(x)}{J_kG(x)}\dint \cH_{n-k}(x)\dint \cH_k(g).
% \end{align}
This allows us to make the change of variable (and also identify the image measure $\pi_*\mu$ on $\G$). Let us choose a function $f$ integrable (w.r.t. $\mu$) on $\bbR^n$ and then, for $p=(g,k,i)$ we have
\begin{align}
\int_{\bbR^n}f(x)\dx=\sum_{k=0}^m\sum_i\int_{I^k_i}\bar f(p)\dint\cH_k(p).
\end{align}
The average function $\bar f$ is defined on $\G$, for $g\in R^k_i$, by 
\be
\bar f(g,k,i)=\int_{G^{-1}(g)\cap R^k_i}\frac{f(x)}{J_kG(x)}\dint \cH_{n-k}(x).
\ee

% \be\label{eq:mes}
% \dint \nu(p)=
% \sum_{k\ieg m}\sum_{i\seg0}\1_{(I_i^k,k,i)}(p)\dint \cH_k(g).
% \ee
% We also define the density $d(p)$ by
% \be\label{eq:densities}
% d(p)=\sum_{k\ieg m}\sum_{i\seg0}\1_{(I_i^k,k,i)}(p)\int_{G^{-1}(g)\cap R^{k}_i}\frac{h(x)}{J_kG(x)}\dint \cH_{n-k}(x).
% \ee
% Then the measure $\dint \pi_*\mu(p)$ could be formally written $\dint \pi_*\mu(p)=d(p)\dint\nu(p)$.

\subsubsection{Projected Dirichlet form}

We have to decompose all the integrals of the left-hand side of Equation \eqref{eq:DF.25} and apply the coarea formula, in order to obtain integrals on the components of $\G$.
For the first integral of Equation \eqref{eq:DF.25}, we obtain
%, by applying Equation \eqref{eq:temp.213}
\begin{align}\label{eq:inte.1}
\int_{\bbR^n}\grad \barr{u}^* a\grad \barr{v}\dint \mu
&=\sum_{k\ieg m}\sum_i\int_{I_i^k}\dint \cH_k(g)\int_{G^{-1}(g)\cap R^{k}_i}\grad \barr{u}(x)^* a(x)\grad \barr{v}(x)h(x)\frac{\dint \cH_{n-k}(x)}{J_kG(x)}.
\end{align}
On the set $R^{k}_i$, we have $\barr u(x)=u(G(x),k,i)$, then let us denote $D_gu$ the derivative of $u$ with respect to the first variable. Note that this derivative has value in the respective tangent space of the submanifold i.e. $\bbR^k$. This is a derivative along the submanifold $I_i^k$ in $\bbR^m$.
% Then, on $R_i^k$, we obtain
% \be
% \grad \barr u(x)=DG^*(x)(D_gu)(G(x),k,i).
% \ee
If we omit the constant quantities $(k,i)$, we get
\begin{align}
(\grad \barr u)^* a\grad \barr v
&=D_gu^*(DGaDG^*)D_gv
\end{align}
Note that since $a=\s\s^*$, the matrix $DGaDG^*=(DG\s)(DG\s)^*$ is also symmetric and non-negative.
Let us also define the average quantity $a_G$ defined on $\G$
\be
a_G((g,k,i))=\int_{G^{-1}(g)\cap R^{k}_i}\frac{(DGaDG^*)h}{J_kG}\dint \cH_{n-k}.
\ee
The integral \eqref{eq:inte.1} becomes
\begin{align}
\int_{\bbR^n}(\grad \barr{u})^* a\grad \barr{v}\dint \mu
&=\sum_{k\ieg m}\sum_i\int_{I_i^k}(D_gu)^* a_GD_gv\dint \cH_k
=\int_{\G}(D_gu)^* a_GD_gv.
\end{align}
%The measure $\dint \nu$ is defined in \eqref{eq:mes}.

The second and third integrals in \eqref{eq:DF.25} are treated accordingly. We denote, for $p=(g,k,i)$
\begin{align}
\Phi_{G}(p)&=\int_{G^{-1}(g)\cap R^{k}_i}\frac{DG\Phi h}{J_kG}\dint \cH_{n-k}\\
F_{G}(p)&=\int_{G^{-1}(g)\cap R^{k}_i}\frac{\grad^*(\Phi h)}{J_kG}\dint \cH_{n-k}.
\end{align}
% Then we get
% \begin{align}
% \int \Phi^*\cro{\barr{u}\grad \barr{v}-\barr{v}\grad \barr{u}}\dint \mu
% &=\int_{\G} \Phi_{G}^*\cro{uD_gv-vD_gu}.
% \end{align}

Therefore, we get the following formulation for the projected Dirichlet form
\begin{align}\label{eq:DF.26}
\cE(u,v)
&=\e\int_{\G}(D_gu)^* a_G D_gv
+\frac12\int_{\G} \Phi_G^*\cro{uD_gv-vD_gu}
+\frac12\int_\G F_{G}uv.
\end{align}

\subsubsection{Identification of the infinitesimal generator}

In order to solve Equation \eqref{eq:DF.25}, we need to write the integrals in Equation \eqref{eq:DF.26} where $v$ does not have any derivative. This can be obtained, at least formally, via some integrations by parts .

% Let us denote $\del I^k_i$ the boundary of $I^k_i$, defined by $\del I^k_i=\clos{I^k_i}\setminus I^h_i$. We also define
% \be
% \Pi(k,i)=\acc{j, I^k_i\subset \del I^{k+1}_j}.
% \ee
% We denote for $u$ defined on $I^k_i$ the trace $T_{(k,i)}u$ on $\del I_i^k$.
% An integration by part over the sets $I^m_j$, which are open domains in $\bbR^m$ is not a problem. However, on the submanifolds $I^k_i$ for $k<m$, an integration by part is a little more subtle since the differentials and the exterior normal must be relative to the submanifold $I^k_i$.
% 
% Then we have, formally,
% \begin{align}\label{eq:temp.45}
% \int_{I_i^k}(D_g u)^* a_G D_g v\dint\cH_{k}
% =\sum_{I_{j}^{k-1}\subset\del I_i^k}\int_{I_{j}^{k-1}}vT_{(k,i)}(a_G D_g u)\cdot n_{ext}\dint\cH_{k-1}
% -\int_{I_i^k} vD^*_g(a_G D_g u)\dint\cH_{k}.
% \end{align}
% $n_{ext}$ denotes the exterior normal (which lives in the tangent space of $I_i^k$ at the boundary point for $k<m$).
% 
% We can do the same for the second part of the middle integral in \eqref{eq:DF.26}
% \begin{align}\label{eq:temp.46}
% \int_{I_i^k} \Phi_G^*u D_g v\dint \cH_k
% &=\sum_{I_{j}^{k-1}\subset\del I_i^k}\int_{I_{j}^{k-1}}uvT_{(k,i)}\Phi_{G}\cdot n_{ext}\dint\cH_{k-1}
% -\int_{I_i^k} vD^*_g(\Phi_{G}u)\dint\cH_{k}.
% \end{align}

We need to identify the integrals for each $I^k_i$. We only have boundary terms if $k<m$, therefore let us first give the generator on the sets $I^m_j$. From Equation \eqref{eq:DF.26}, %after inserting Equations \eqref{eq:temp.45} and \eqref{eq:temp.46}, 
we obtain on each $I^m_j$,
\be\label{eq:gene}
\cL_\e u=\frac1{d}
\cro{\e D^*_g(a_G D_g u)+(\Phi_{G})^*D_gu+\frac12\pare{D^*_g\Phi_{G}-F_G}u}.
\ee
Note that since all sets $I_i^k$, for $k<m$, are contained in the boundary of a domain $I_j^m$.%, Equation \eqref{eq:gene} defines a trace $T_{(m,j)}(\cL_\e u)$ by prolongation on each closure of $I_j^m$. 

Collecting the boundary term along every submanifold gives gluing conditions, this is the same process as the $2$-dimensional case. Heuristically, these gluing conditions connect the value of the generator restricted on a submanifold $I^k_i$ to the limiting values of the generator at the boundary of the submanifolds $I^{k+1}_j$ such that $\del I^{k+1}_j\subset I^k_i$.

\begin{remark}
For an explicit computation of the generator, defined by Equation \eqref{eq:gene}, one would need a lemma analogous to the Lemma \ref{lem:derive} to compute the derivatives of the  averaging $a_G$ and $\Phi_G$.
\end{remark}


\begin{thebibliography}{99}

\bibitem{evans.gariepy92}{Evans, Lawrence C. and Gariepy, Ronald F.}: {\it Measure theory and fine properties of functions}, {Studies in Advanced Mathematics}, {CRC Press}, {1992}.


\bibitem{feller54} {William Feller}: {\it Diffusion processes in one dimension}, {Trans. Amer. Math. Soc.}, 77, {1--31}, {1954}.

\bibitem{freidlin.sheu00} {Freidlin, Mark and Sheu, Shuenn-Jyi}: 
{\it Diffusion processes on graphs: stochastic differential equations, large deviation principle}, {Probab. Theory Related Fields}, 116(2), 181--220,  {2000}.


\bibitem{freidlin.weber04} {Freidlin, Mark and Weber, Matthias}:
{\it Random perturbations of dynamical systems and diffusion processes with conservation laws}, {Probab. Theory Related Fields}, {128(3)}, {441--466}, {2004}.

\bibitem{freidlin.weber01} {Freidlin, Mark and Weber, Matthias}: {\it On random perturbations of {H}amiltonian systems with many degrees of freedom}, {Stochastic Process. Appl.}, 94(2), 199--239, {2001}.

\bibitem{freidlin.wentzell12} {Freidlin, Mark I. and Wentzell, Alexander D.}: {\it Random perturbations of dynamical systems},
{Springer}, {2012}.

\bibitem{freidlin.wentzell04} {Freidlin, M. I. and Wentzell, A. D.}: {\it Diffusion processes on an open book and the averaging principle}, {Stochastic Process. Appl.}, 113(1), 101--126, 2004.

\bibitem{freidlin.wentzell94} {Freidlin, Mark I. and Wentzell, Alexander D.}: {\it Random perturbations of {H}amiltonian systems}, {Mem. Amer. Math. Soc.}, 109(523), {1994}.

\bibitem{freidlin.wentzell93}  {Freidlin, Mark I. and Wentzell, Alexander D.}: {\it Diffusion processes on graphs and the averaging principle}, {Ann. Probab.}, 21(4), 2215--2245, {1993}.

\bibitem{hino98} {Hino, Masanori}: {\it Convergence of non-symmetric forms}, {J. Math. Kyoto Univ.}, 38(2), 329--341, 1998.

\bibitem{jacod.shiryaev} {Jacod, Jean and Shiryaev, Albert N.}:
{\it Limit theorems for stochastic processes}, {Springer-Verlag}, {1987}.

\bibitem{KKVW09} {Kant, Ulrike and Klauss, Tobias and Voigt, J{\"u}rgen and Weber, Matthias}: 
{\it Dirichlet forms for singular one-dimensional operators and on graphs}, {J. Evol. Equ.}, 9(4), 637--659, {2009}.

\bibitem{kolesnikov05} {Kolesnikov, Alexander V.}: 
{\it Convergence of {D}irichlet forms with changing speed measures on {$\Bbb R^d$}}, {Forum Math.}, 17(2), 225--259, {2005}.

\bibitem{KPS12}
{Kostrykin, Vadim and Potthoff, J{\"u}rgen and Schrader, Robert}:
{\it Brownian motions on metric graphs}, {J. Math. Phys.}, 53(9), 2012.

\bibitem{krantz.parks08} {Krantz, Steven G. and Parks, Harold R.}: {\it Geometric integration theory}, {Birkh\"auser Boston Inc.}, {2008}.

\bibitem{kuwae.shioya03} {Kuwae, Kazuhiro and Shioya, Takashi}: 
{\it Convergence of spectral structures: a functional analytic theory and its applications to spectral geometry}, {Comm. Anal. Geom.}, 11(4), 599--673, 2003.

\bibitem{ma.rockner92} {Ma, Zhi Ming and R{\"o}ckner, Michael}:
{\it {Introduction to the theory of (nonsymmetric) {D}irichlet forms}}, {Springer-Verlag}, 1992.

\bibitem{mandl68} {Mandl, Petr}: {\it Analytical treatment of one-dimensional {M}arkov processes}, {Academia Publishing House of the Czechoslovak Academy of Sciences, Prague}, {1968}.

\bibitem{oshima13}{Oshima, Yoichi}, {\it Semi-{D}irichlet forms and {M}arkov processes}, {Walter de Gruyter \& Co., Berlin}, {2013}.
		

\bibitem{stroock.varadhan79} {Stroock, Daniel W. and Varadhan, S. R. Srinivasa}: {\it {Multidimensional diffusion processes}},
{Springer-Verlag}, {1979}.

\bibitem{toelle06} {T\"olle, Jonas}: {\it Convergence of non-symmetric forms with changing reference measure}, {University Bielefeld}, {2006}.

\end{thebibliography}
\end{document}